\documentclass[11pt,a4paper]{article}
\usepackage[utf8]{inputenc}
\usepackage[english]{babel}
\usepackage[T1]{fontenc}
\usepackage{amsmath}
\usepackage{amsfonts}
\usepackage{amssymb,amsthm}
\usepackage{fullpage}
\usepackage{graphicx}
\usepackage{xcolor}
\usepackage{bm}
\usepackage{hyperref}
\usepackage{authblk}
\usepackage{mathtools}
\usepackage[normalem]{ulem}

\DeclarePairedDelimiter\floor{\lfloor}{\rfloor}

\newcommand{\ms}{\mathrm{MS}}
\usepackage{bm}

\newcommand{\x}{{\bm x}}

\newcommand{\R}{\mathbb{R}}

\newcommand{\N}{\mathbb{N}}

\newcommand{\betad}{{\beta_d}}

\def\AA{{\bm A}}
\def\vi{v}
\def\vii{{\bm v}}
\def\bX{{\bm X}}

\def\Kor{\mathcal{K}}

\newtheorem{theorem}{Theorem}[section]
\newtheorem{lemma}[theorem]{Lemma}

\newtheorem{proposition}[theorem]{Proposition}
\newtheorem{corollary}[theorem]{Corollary}
\newtheorem{assumption}[theorem]{Assumption}

\theoremstyle{remark}
\newtheorem{remark}[theorem]{Remark}

\numberwithin{equation}{section}

\title{Polynomial Approximation of Symmetric Functions}

\author[1]{Markus Bachmayr}
\affil[1]{Institut f\"ur Geometrie und Praktische Mathematik, RWTH Aachen University, Templergraben 55, 52056 Aachen, Germany}

\author[2]{Genevi\`eve Dusson\footnote{genevieve.dusson@math.cnrs.fr}}
\affil[2]{Laboratoire de Math\'ematiques de Besan\c{c}on,
  UMR CNRS 6623,
  Universit\'e Bourgogne Franche-Comt\'e,
  16~route de~Gray,
  25030~Besan\c{c}on,
  France}
  
\author[3]{Christoph Ortner}
\affil[3]{Department of Mathematics, University of British Columbia, 1984 Mathematics Road,
Vancouver, BC, V6T 1Z2, Canada}

\author[4]{Jack Thomas}
\affil[4]{Mathematics Institute, Zeeman Building, University of Warwick, CV4 7AL, UK}

\begin{document}
\maketitle

\abstract{
    We study the polynomial approximation of symmetric multivariate functions and of multi-set functions. Specifically, we consider $f(x_1, \dots, x_N)$, where $x_i \in \mathbb{R}^d$, and $f$ is invariant under permutations of its $N$ arguments. We demonstrate how these symmetries can be exploited to improve the cost versus error ratio in a polynomial approximation of the function $f$, and in particular study the dependence of that ratio on $d, N$ and the polynomial degree. These results are then used to construct approximations and prove approximation rates for functions defined on multi-sets where $N$ becomes a parameter of the input.
}


\section{Introduction}

Many quantities of interest in sciences and engineering  exhibit symmetries. The approximation of such quantities can be made more efficient when these symmetries are correctly exploited.
A typical problem we have in mind is the approximation of particle models, in particular interatomic potentials, or Hamiltonians in quantum chemistry, materials science or bio-chemistry.
Similar symmetries are present in $n$-point correlation or cumulant functions of stochastic processes.
Our current work focuses on polynomial approximation of multivariate functions that are symmetric under arbitrary permutations of coordinates, as a foundation for more complex symmetries. Even though our focus in the present work is on {\em symmetric} functions, our results are also directly relevant for {\em anti-symmetric} wave functions~\cite{acewave2022}: first, our complexity estimates immediately apply to anti-symmetric parameterisations; and secondly, many successful nonlinear parameterisations of wave functions (e.g., Jastrow or backflow) are in terms of symmetric components.

The performance of an approximation can be measured in different ways. One can measure the number of degrees of freedom required to achieve a given accuracy, in a target norm.  A second important factor is the evaluation cost of the approximate function. Indeed, one may reduce the number of basis functions required to approximate a given function within a given tolerance while greatly increasing the evaluation cost. Therefore, a good compromise between these two aspects is particularly important.

Thus, the aim of this article is to show how the approximation of permutation invariant functions can be made efficient by combining two elements: First, a particular symmetrisation in the evaluation of the function leading to a linear evaluation cost of each basis function with respect to the number of variables; and second, profiting from the symmetries to speed the convergence of the approximation with respect to the number of basis functions and evaluation cost. 

For physical models, one should also incorporate isometry invariance into the analysis, however this would make the analysis significantly more complex while only marginally improving our results. Thus for the sake of simplicity, the present work will only consider permutation invariant functions (symmetric functions) and we refer to \cite{Yutsis1965-rr, Bachmayr2019-ec} to explain how invariance under $O(d)$ (or other groups) can in principle also be incorporated into our framework.

In general, multivariate approximation of functions by polynomials on product domains (see for example~\cite{Mason1980-li,GriebelOettershagen:16,Trefethen2017-rc}) is -- even for analytic approximands -- subject to the \emph{curse of dimensionality}: the number of parameters necessary to reach a given accuracy increases exponentially with the space dimension.
One possibility to avoid this effect when approximating smooth high-dimensional functions is to exploit anisotropy in the different dimensions~\cite{Cohen2015-ol,GriebelOettershagen:16}.
In the present setting, however, due to the symmetry all dimensions play an equivalent role. Approximations of symmetric functions have been studied in~\cite{Han2019-ae} in a context of nonlinear approximation, where the authors provide bounds on the number of parameters needed to approximate symmetric or anti-symmetric functions within a target accuracy. These results also exhibit the curse of dimensionality.
However, note that due to the two different sets of assumptions used in~\cite{Han2019-ae} and our work, it is not straightforward to directly compare the results.
Related to our own work is the study of deep set architectures \cite{Zaheer2017:deepsets, Qi2017} for approximating set-functions, where the symmetry of the approximation is enforced by summation in a latent space. Theoretical results on deep sets \cite{Zaheer2017:deepsets,Wagstaff2019} relate the maximum number of elements in the set input to the dimension of the latent space (which is analogous to the total degree of our approximation) that is required to represent the function exactly, but no error estimates are available so far. 

In the present work we develop a rigorous approximation theory for symmetric (and multi-set) functions that directly relates the number of parameters to the approximation error, {\em independently of the space dimension} (or the number of inputs). We also note that such efficient representations of symmetric functions can heavily rely on invariant theory~\cite{Derksen2015-km} and group theory~\cite{sagan2001symmetric}. 
We will start in Section~\ref{sec:1d} by considering functions with full permutation symmetry, that is, invariance with respect to the permutation of any two variables. In Section~\ref{sec:results}, we will provide generic results for functions exhibiting symmetry with respect to permutations of vectorial arguments $(\R^d)^N$, which is the most typical situation for physical models. Here, $d$ is the dimension of each argument $x_i$, which could for example represent the position or momentum of a particle, while $N$ denotes the number of such arguments. Our analysis is particularly concerned with the question of how the two dimensions $d, N$ and the polynomial degree $D$ are connected in terms of approximation error and computational cost. 
In this regard, our point of view differs from the one of independent accuracy and dimensionality parameters, as taken for instance in \cite{Weimar12,Han2019-ae}.

Our primary motivation for this study is as a foundation for the approximation of extremely high-dimensional functions and of multi-set functions which can be decomposed in terms of a body-order expansion. The multi-set function setting is particularly interesting for us since it commonly arises in the representation of particle interactions. We will show in Section~\ref{sec:mset} how to extend our symmetric function approximation results to obtain approximation rates for functions defined on multi-sets. 
In that setting we will have to address the simultaneous approximation of a family of related functions with increasing dimensionality. Since our framework has close connections with deepsets, our analysis may also shed new light on those architectures. We briefly comment on the connection in Remark~\ref{rem:deepsets}, but do not include a deeper exploration in the present work.

\section{Symmetric Functions in $\R^N$}
\label{sec:1d}
Before we formulate our most general results we consider the approximation of a smooth symmetric function $f : [-1, 1]^N \to \R$. By $f$ being {\em symmetric} we mean that 
\begin{equation} \label{eq:sym1}
  f(x_{\sigma 1}, \dots, x_{\sigma N}) = f(x_1, \dots, x_N) \qquad \forall 
  \x \in [-1,1]^N, \quad \sigma \in {\rm Sym}(N),
\end{equation}
with ${\rm Sym}(N)$ denoting the symmetric group of degree $N$. For later reference we define 
\[
  C_{\rm sym}([-1,1]^N) := \big\{ f \in C([-1,1]^N) \colon 
              f \text{ is symmetric} \big\}.
\]

In this section we will outline the main ideas how symmetry can be optimally incorporated into approximation schemes in the simplest possible concrete setting, but will then generalize them in various ways in \S~\ref{sec:results}.

A general $f \in C([-1,1]^N)$ can be expanded as a Chebyshev series, 
\[
  f(\x) = \sum_{\vii \in \N^N } \hat{f}_\vii T_\vii(\x),
\]
where $T_\vii = \otimes_{n = 1}^N T_{\vi_n}$ and $T_\vi$ are the standard Chebyshev polynomials of the first kind. To allow for the possibility of constructing sparse approximations we will assume that $f$ belongs to a Korobov class, 
\begin{equation}  \label{eq:cheb_coeffs_analyticity}
  \Kor(M, \mu, \rho) := 
  \big\{
      f \text{~s.t.~} {\textstyle \sum_{\vii \in \N^N}} \rho^{\|\vii\|_1} |\hat{f}_\vii| 
         \leq M \mu^N
  \big\},
\end{equation}
which one can justify, e.g., through a suitable multi-variate generalisation of analyticity. The dependence of the upper bound on $N$ also arises naturally in this context. 
 
In high-dimensional approximation one often considers spaces with weighted norms that reflect the relative importance of dimensions, which would lead to replacing $\|\vii\|_1$ by accordingly weighted quantities; see, e.g., Chapter 5 of \cite{NovakWozniakowski08} and the references given there. However, due to our assumption of symmetry, all dimensions are equally important. 
In this case, the following total degree approximation is natural: For $D > 0$, define 
\begin{equation} \label{eq:cheb_total_deg_approx_nosym}
  f_D(\x) := \sum_{\vii : \|\vii\|_1 \leq D} \hat{f}_\vii T_\vii(\x), 
\end{equation}
then we immediately obtain the exponential approximation error estimate 
\begin{equation} \label{eq:cheb_approx_nosym}
  \| f - f_D \|_{L^\infty} \leq M \mu^N \rho^{-D}.
\end{equation}

Although the term $\rho^{-D}$ suggests an excellent approximation rate, the curse of dimensionality still enters through the prefactor $\mu^N$ as well as through the cost of evaluating $f_D$, which scales as 
\[
  \binom{N+D}{D} \sim
  \begin{cases}
      D^N / N!, & \text{as } D \to \infty, \\ 
      N^D / D!, & \text{as } N \to \infty.
  \end{cases}
\]
where $\binom{N+D}{D}$ is the number of terms in \eqref{eq:cheb_approx_nosym}. The number of operations involved in each term can be reduced to $\mathcal{O}(1)$; cf Remark.~\ref{rmk:cost}. 
Although this is far superior to the naive tensor product approximation leading to $\mathcal{O}(D^N)$ terms, it remains expensive in high dimensions.

Incorporating the symmetry into the parameterisation allows us to significantly reduce the number of basis functions (or, parameters): If we require that $f_D$ inherits the symmetry~\eqref{eq:sym1} -- see~\cite[Section 3]{Bachmayr2019-ec} why this does not worsen the approximation quality -- one can readily see that it is reflected in the coefficients via 
\[
  \hat{f}_{\vii} = \hat{f}_{\sigma(\vii)} \qquad \forall \sigma \in {\rm Sym}(N).
\]
Thus, we can obtain a symmetrised representation 
\begin{align} \label{eq:f_D_sym_naive}
  f_D(\x) &= 
  \sum_{\vii \in \N^N_{\rm ord}: \| \vii \|_1 \leq D}
      c_{\vii}  \;
      \operatorname{sym} T_{\vii}(\x) \\ 
  \text{where} \qquad & 
  \operatorname{sym} T_{\vii}(\x) = 
  \sum_{\sigma \in {\rm Sym}(N)} T_{\vii}(\sigma\x)
  =    
  \sum_{\sigma \in {\rm Sym}(N)} T_{\sigma \vii}(\x)
\end{align}
and $\N^N_{\rm ord}$ denotes the set of all {\em ordered} $N$-tuples, i.e.,
$\vii \in \N^N_{\rm ord}$ if $\vi_1 \leq \vi_2 \leq \dots \leq \vi_N$. When
$\vii$ is not strictly ordered then $\sigma\vii = \vii$ for some permutations
and hence the coefficient $c_{\vii}$ is different from $\hat{f}_{\vii}$.

It is immediate to see that $\operatorname{sym} T_\vii$ form a basis of the space of symmetric polynomials, which in turn is dense in $C_{\rm sym}$; see~\cite[Proposition 1]{Bachmayr2019-ec} for more details.

Although the representation \eqref{eq:f_D_sym_naive} significantly reduces the number of parameters $c_\vii$ (almost by a factor $N!$), it {\em does not} reduce the cost of evaluating $f_D$ due to the $N!$ cost of evaluating each symmetrised basis function $T^{\rm sym}_{\vii}$. However, an elementary idea from invariant theory leads to an alternative symmetric basis, and hence an alternative scheme to evaluate $f_D$, which significantly lowers the cost and appears to entirely overcome the curse of dimensionality: It is a classical result that, since $f_D$ is a symmetric polynomial, it can be written in the form
\begin{equation} \label{eq:approx_powersum_polys}
  f_D(\x) = q_D(p_1, \dots, p_N), 
\end{equation}
where $p_n(\x) := \sum_{j = 1}^N x_j^n$ are the power sum polynomials. This representation fully exploits the symmetry and one could expand on this idea to construct an efficient evaluation scheme. 

Here, we follow a closely related construction, inspired by~\cite{Drautz2019-rb}, which is easier to analyze and most importantly to generalize to more complex scenarios (cf. \S~\ref{sec:results}). A straightforward generalisation of the power sum polynomials is the following symmetric one-body basis, 
\begin{equation} \label{eq:1d:onebodybasis}
  A_\vi(\x) := \sum_{n = 1}^N T_\vi(x_n), \qquad \vi \in \N. 
\end{equation}
\begin{remark}
    Since $\big( T_v(x_n) \big)_{n=1}^N$ represents a feature vector, (\ref{eq:1d:onebodybasis}) is a pooling operation analogous to that introduced in \cite{Zaheer2017:deepsets} to impose symmetry in deep set architectures. We will say more about this connection in Remark~\ref{rem:deepsets}.
\end{remark}
Indeed, $A_1, \dots, A_N$ could play the same role as $p_1, \dots, p_N$ in \eqref{eq:approx_powersum_polys}, however, we use them differently by forming the products 
\begin{equation}
  \label{eq:1d:Abasis}
  \AA_\vii(\x) := \prod_{t = 1}^N A_{\vi_t}(\x), \qquad \vii \in \N^N.
\end{equation}
If $\sigma(\vii)$ is a permutation of $\vii$ then the two products $\AA_{\vii}$ and $\AA_{\sigma(\vii)}$ coincide; that is, $\AA_\vii$ are again symmetric polynomials. In fact they form a complete basis of that space. 

\begin{lemma} \label{th:AA_is_a_basis}
  The set $\AA := \{ \AA_\vii \colon \, \vii \in \N^N_{\rm ord} \}$ 
  is a complete basis of the space of symmetric polynomials.
\end{lemma}
\begin{proof}
  We begin by rewriting the naive symmetrised basis as 
  \begin{align*}
      \operatorname{sym} T_{\vii}(\x) = \sum_{\sigma \in {\rm Sym}(N)} \prod_{t = 1}^N T_{\vi_t}(x_{\sigma t})
      &= 
      \frac{1}{N!} \sum_{j_1 \neq \cdots \neq j_N} \prod_{t = 1}^N T_{\vi_t}(x_{j_t}) \\ 
      &= 
      \frac{1}{N!} \sum_{j_1, \dots, j_N}  
      \prod_{t = 1}^N T_{\vi_t}(x_{j_t}) 
      + B' \\ 
      &= 
      \frac{1}{N!} \AA_{\vii}(\x) + B',
  \end{align*}
  where $B'$ contains the ``self-interactions'', i.e., repeated $j_t$ indices. Therefore, 
  \[
      B'  \in \text{span} \bigg\{ \sum_{\sigma \in {\rm Sym}(N)}  \prod_{j=1}^N T_{v_j}(x_{\sigma j})\,\colon \;
      \text{at least one $v_j$ is 0}
      \bigg\}.
  \]
  Since $T_0 = 1,$ $B'$ only involves terms with strictly less than $N$ products. Therefore, the change of basis from $T^{\rm sym}_\vii$ to $\AA_\vii$ is lower triangular with $1/N!$ terms on the diagonal, and is hence invertible.
\end{proof}

A second immediate observation is that the total degree $\|\vii\|_1$ of a tensor product basis function $\otimes_{t = 1}^N T_{\vi_t}$ immediately translates to the $\AA$ basis. That is, the total degree of $\AA_\vii$ is again $\|\vii\|_1$, which yields the next result. 

\begin{corollary}
  There exist coefficients $\tilde{c}_\vii$ such that 
  \begin{equation} \label{eq:cheb_totaldeg_withsym}
      f_D(\x) = 
      \sum_{\vii \in \N^N_{\rm ord}: \| \vii \|_1 \leq D } \tilde{c}_{\vii} \AA_{\vii}(\x).
  \end{equation}
\end{corollary}

In Remark~\ref{rmk:cost} we explain that the computational cost of evaluating \eqref{eq:cheb_totaldeg_withsym} is directly proportional to the the number of terms, or parameters, which we denote by 
\[
  P(N, D) := 
  \# \big\{ \vii \in \N^N_{\rm ord}: \| \vii \|_1 \leq D \big\}.
\]
When clear from the context we will write $P = P(N,D)$. To estimate that number we observe that the set 
\[
  \big\{ \vii \in \N^N_{\rm ord}: \| \vii \|_1 = D \big\}
\]
can be interpreted as the set of all integer partitions of $D$, of length at most $N$ (indices $\vi = 0$ do not contribute). There exist various bounds for the number of such partitions that incorporate both $N$ and $D$, such as~\cite[Theorem 4.9.2]{Ramirez_Alfonsin2005-et}, originally presented in~\cite{Beged-dov1972-mp},
\begin{equation} \label{eq:params_regimes}
  P(N, D) \le \frac{\left( D+\frac{N(N+1)}{2}\right)^N}{(N!)^2}
  \sim \frac{D^{N}}{(N!)^2} 
  \qquad \text{as } D \rightarrow \infty,
\end{equation}
which (unsurprisingly) suggests that we gain an additional factor $N!$ in the number of parameters and in the computational cost, compared to the total degree approximation which has asymptotic cost $\binom{N+D}{D} \sim \frac{D^N}{N!}$ as $D \to \infty$. We will return to these estimates below.

Since we are particularly interested in an $N$-independent bound we will instead use a classical result of Hardy and Ramanujan~\cite{Hardy1918-gg}.

\begin{lemma} \label{th:hardyramanujan}
  For any $N, D$ we have 
  \begin{equation} \label{eq:hardyramanujan}
      P(N, D) 
      \leq 
      \frac{1}{8 \sqrt{3} D } \exp \left(\pi \sqrt{\textstyle{\frac{4}{3}} D}  \right).
  \end{equation}
\end{lemma}
\begin{proof}
  The result of \cite[Theorem 6.8]{2011-ip} states that, the cardinality of the set $\big\{ \vii \in \N^N_{\rm ord}: \| \vii \|_1 \le D \big\}$ is bounded by the number of additive integer partitions of $2D$, i.e., 
  \[ 
      P(N,D) = 
      \#\big\{ \vii \in \N^N_{\rm ord}: \| \vii \|_1 \le D \big\}
      \leq 
      \#\big\{ \vii \in \N^N_{\rm ord}: \| \vii \|_1 = 2D \big\}.
  \]
  The result of Hardy and Ramanujan~\cite{Hardy1918-gg} estimates the latter cardinality as stated 
  in \eqref{eq:hardyramanujan}.
\end{proof}

The key property of this bound is that it is independent of the domain dimension $N$, which yields the following super-algebraic (but sub-exponential) convergence rate.

\begin{theorem} \label{th:mainresult_1d}
  Let $f \in C_{\rm sym}([-1,1]^N) \cap \Kor(M, \mu, \rho)$, then there exists a constant $c > 0$ such that for all $D \geq c N$, the symmetric total degree approximation \eqref{eq:cheb_totaldeg_withsym} satisfies
  \[
      \|f - f_D \|_{L^\infty} 
      \leq C \exp\Big( - \alpha [\log P]^2\Big),
  \]
  where $C, \alpha > 0$ are independent of $N$ and $D$ but may depend on 
  $M, \mu, \rho$.
\end{theorem}
\begin{proof}
  From our foregoing discussion we obtain that 
  \[
      \log \|f - f_D \|_{L^\infty} 
      \leq \log M + N \log \mu - D \log \rho. 
  \]
    Fix any $1 < \bar\rho < \rho$, then clearly, there exists $c>0$ such that $D \geq cN$ implies 
  \[
      \log M + N \log \mu - D \log \rho \leq \log M - D \log \bar\rho,
  \]
  hence, in this regime, we obtain 
  \begin{equation} \label{eq:prf1d:20}
      \|f - f_D \|_{L^\infty}  \leq
      M \bar\rho^{-D}.
  \end{equation}
  Next, we estimate $D$ in terms of the number of parameters, 
  using Lemma~\ref{th:hardyramanujan}, by 
      $\log P  
      \leq C_1 \sqrt{D}$,   
  which we invert to obtain 
  \[ 
      D \geq \bigg[\frac{\log P}{C_1} \bigg]^2.
  \]
  Inserting this into \eqref{eq:prf1d:20} completes the proof with 
  $C = M$ and $\alpha = \log(\bar\rho) / C_1^2$.
\end{proof}

\begin{remark}
  Although the proof of Theorem~\ref{th:mainresult_1d} appears to sacrifice a significant amount of information by using the Hardy-Ramanujan result instead of the sharper estimate \eqref{eq:params_regimes}, it turns out to be sharp in the regime $c_1 N \leq D \leq C_1 N^2$.
  We will explain this observation in more detail in \S~\ref{sec:sharpness}.
\end{remark}

\section{General Results}
\label{sec:results}
\subsection{A basis for symmetric functions in $(\R^d)^N$}
We consider the approximation of functions $f(\x_1, \dots, \x_N)$ where each coordinate $\x_j \in \Omega \subset \R^d$, $d \in \mathbb{N}$, and $f$ is invariant under permutations, i.e., 
\begin{equation}
  \label{eq:perm_symm}
  f(\x_1, \dots, \x_N) = f(\x_{\sigma 1}, \dots, \x_{\sigma N}) 
  \qquad \forall \sigma \in {\rm Sym}(N).
\end{equation}
We will indicate this scenario by saying that $f : \Omega^N \to \R$ is symmetric. We assume throughout that $\Omega$ is the closure of a domain, and define the space 
\[
  C_{\rm sym}(\Omega^N) := \{ f \in C(\Omega^N)\,\colon\; f \text{ is symmetric}\}.
\]
To construct approximants we begin from a one-body basis, 
\[
  \Phi := \big\{ \phi_\vi \,\colon\; \vi \in \N \big\}.
\]

\begin{assumption} \label{as:PHI}
  To enable the convenient extension of uni-variate to multi-variate constructions we make the following assumptions which are easily justified for all basis sets that we have in mind; see~Remark~\ref{rem:examples}.
\begin{itemize}
  \item[{\rm($\Phi$1)}] The set of basis functions $\Phi$ is dense in $C(\Omega)$. 
  \item[{\rm($\Phi$2)}] Each $\phi_v \in \Phi$ has an associated degree $\operatorname{deg} \phi_v \in \mathbb{N}$.
  \item[{\rm($\Phi$3)}] $v = 0$ is an admissible index, $\phi_0 = 1$ and $\operatorname{deg} \phi_0  = 0$. 
\end{itemize}
  
  \noindent Finally, we need a definition of ``intrinsic dimensionality'' of our basis, and for simplicity also require that it matches the dimensionality of the domain $\Omega$. We therefore make the following additional assumption: 
\begin{itemize}
  \item[{\rm($\Phi$4)}] The number of one-body basis functions $\phi_v$ of degree $i$ is bounded by
  \[
      c_d(i) \le c (i+d-1) (i+d-2) \ldots (i+1), 
  \]
  where $c>0$ is a constant.
\end{itemize}
\end{assumption}

\noindent Under assumption ($\Phi$1), the tensor product basis functions 
\[
  \Phi^{\otimes N} := \big\{ \phi_{\vii} = \otimes_{t = 1}^N \phi_{v_t}
         \colon\, \vii \in \N^N \big\}
\] 
then satisfy that $\operatorname{span} \Phi^{\otimes N}$ is dense in $C(\Omega^N)$ by the characterization of this space as an $N$-fold injective tensor product (see, e.g., \cite[\S 3.2]{Ryan2002}). As an immediate consequence, we obtain that the symmetrised tensor products, 
\[
  \operatorname{sym}\big[\Phi^{\otimes N}\big]
  :=
  \bigg\{ 
      \operatorname{sym}\big[\phi_{\vii}\big] := \frac{1}{N!} \sum_{\sigma \in {\rm Sym}(N)}
         \phi_{\vii} \circ \sigma 
         \,\colon\; 
      \vii \in \N^N 
  \bigg\}
\]
span the symmetric functions, that is, we have the following result.

\begin{proposition}
  ${\rm span}\,\operatorname{sym}\big[\Phi^{\otimes N}\big]$ is dense in $C_{\rm sym}(\Omega^N)$, and
      \[
        \| f - \operatorname{sym}[f_D] \| \leq \| f - f_D \|.
      \]
\end{proposition}
\begin{proof}
  The argument is a verbatim repetition of the proof of  the argument in \S~\ref{sec:1d}, Lemma~\ref{th:AA_is_a_basis}, and of \cite[Sec. 3]{Bachmayr2019-ec}.
\end{proof}

An alternative symmetric many-body basis can be constructed by mimicking the construction of the power sum polynomials in \S~\ref{sec:1d} (see also \cite{Drautz2019-rb,Bachmayr2019-ec} which are our inspiration for this construction), 
\begin{align}
  \AA_\vii(\bX) &= \prod_{t = 1}^N A_{\vi_t}, \qquad \text{where} \\ 
  A_\vi(\bX) &= \sum_{n = 1}^N \phi_\vi(\x_n).
\end{align}
We denote the corresponding basis set by 
\[
  \AA := \big\{ \AA_{\vii} \,\colon\; \vii \in \mathbb{N}^N_{\rm ord} \}, 
\]
where we recall that $\N^N_{\rm ord} := \{ \vii \in \N^N \,\colon\; \vi_1 \leq \vi_2 \leq \dots \}$. For future reference we define 
\[
    {\rm deg}(\vii) := \sum_{t = 1}^N \operatorname{deg} \phi_{v_t}, \qquad \text{for } 
    \vii \in \mathbb{N}^N.
\]

\begin{proposition}
  ${\rm span} \,\AA$ = ${\rm span}\,\operatorname{sym}\big[\Phi^{\otimes N}\big]$, and 
  in particular, ${\rm span}\,\AA$ is dense in $C_{\rm sym}(\Omega^N)$. 
  Moreover, $\AA$ is linearly independent.
\end{proposition}
\begin{proof}
  The argument is identical to the $d = 1$ case. 
\end{proof}

\begin{remark}[Examples of Basis Sets]
\label{rem:examples}
{\it (i) Tensor product Chebyshev polynomials. }
The first concrete case we consider is $\Omega = [-1, 1]^d$ with one-particle basis $\Phi := \big\{ T_{\bm k} := \otimes_{i = 1}^d T_{k_i}
                  \,\colon\, 
                  {\bm k} \in \N^d \big\}.
$
This is the simplest setting to which our analysis applies. In this case, we would simply take $\operatorname{deg} T_k = k$.

{\it (ii) Two-dimensional rotational symmetry. }
Assuming a rotationally symmetric domain,  
 $\Omega := \{ x \in \R^2 \,\colon \, r_0 \leq |x| \leq r_1 \}$,
a natural one-body basis is
\[
  \Phi := \big\{ \phi_{n m}(x) = R_{nm}(|x|) e^{i m\, {\arg x}}
          \, \colon\, 
            n \in \N, m \in \mathbb{Z} \big\},
\]
where the radial basis $R_n$ or $R_{nm}$ could, e.g., be chosen as transformed Zernike, Chebyshev, or other orthogonal polynomials. This has the natural associated total degree definition 
${\rm deg} \phi_{nm}  := n + |m|.$

{\it (iii) Three-dimensional rotational symmetry. }
Our final example concerns three-dimensional spherical symmetry, where the one-particle domain is given by $\Omega := \{ x \in \R^3 \,\colon\, r_0 \leq |x| \leq r_1 \}.$
In this case, it is natural to employ spherical harmonics to discretize the angular component, i.e., 
\[
  \Phi := \big\{ \phi_{nlm}(x) := R_{nl}(|x|) Y_l^m(\hat{x}) 
         \,\colon\,  n, l \in \N, m = -l, \dots, l \big\},
\]
with associated degree ${\rm deg} \phi_{nlm}  := n + l + |m|.$
Since $|m| \leq l$ and since $Y_l^m$ are coupled under rotations it is often more natural though to take $\operatorname{deg} \phi_{nlm} = n + l$, which does not change our results. 
\end{remark}

\begin{remark}[Lie Group Symmetries]
    In many physical modelling scenarios one would also like to incorporate Lie group symmetries, e.g. invariance or covariance under rotations and reflections, or under in relativity theory under the Lorentz group. This can in principle be incorporated into our analysis, but there in general the coupling between the permutation group and such Lie groups is non-trivial and it becomes difficult to give sharp results. We therefore postpone such an analysis to future works. 
    
    That said, such additional symmetry does not reduce the parameterisation complexity nearly as much as the permutation symmetry does. For example in case (ii) of the previous remark, if rotational symmetry were imposed then this would only lead  to an additional constraint $\sum_{t = 1}^N m_t = 0$ on the basis functions. Secondly, such structural sparsity changes our remarks (e.g., Remark~\ref{rmk:cost}) on computational complexity. We show in \cite{Kall2021} that asymptotically optimal evaluation algorithms can still be constructed for those cases.
\end{remark}

\begin{remark}[Computational Cost]
\label{rmk:cost}
  We briefly justify our claim that the computational cost for both \eqref{eq:cheb_total_deg_approx_nosym} and \eqref{eq:cheb_totaldeg_withsym} is directly proportional to the number of parameters. For the sake of brevity we focus only on the evaluation \eqref{eq:cheb_totaldeg_withsym} in two stages. First, we compute the symmetric one-body basis $A_\vi$ defined in \eqref{eq:1d:onebodybasis}, where $\vi = 0, \dots, \max_{\phi_v \in \Phi} \operatorname{deg}(\phi_v)$. According to {\rm($\Phi 4$)} there can be at most $c_d(D)$ elements of that basis. We shall assume that the cost of evaluating it scales as $\mathcal{O}(N D^d)$, which can for example be achieved by suitable recursive evaluations of a polynomial basis. 

Once the one-body basis $A_\vi$ is precomputed and stored, the products $\AA_\vii$ can be evaluated recursively. Assume that the multi-indices are stored in a lexicographically ordered list. As we traverse the list we store the computed basis functions $\AA_{\vii}$. For each new multi-index in the list, 
\[
  \vii = (\vi_1, \dots, \vi_N), 
\] 
we express the associated basis function as 
\[
  \AA_{\vii} = N^{-1} A_{\vi_1} \AA_{0, \vi_2, \dots, \vi_N},
\]
thus evaluating  $\AA_{\vii}$ with $\mathcal{O}(1)$ cost. Note that the multi-indices $\vii$ in the total degree approximations are a downset, hence if $(\vi_1, \dots, \vi_N)$ is part of the approximation then the new multi-index $(0, \vi_2, \dots, \vi_N)$ is also part of the approximation. In summary we therefore obtain that the total cost in evaluating \eqref{eq:cheb_totaldeg_withsym} is bounded, up to a constant, by
\[
  N D^d + P(N, D).
\]
\end{remark}

\subsection{Approximation errors}
Our starting point in the $d = 1$ case was to relate analyticity of the target function $f$ to approximation rates for the Chebyshev expansion. The same idea can be applied here but would be restrictive. In the multi-dimensional setting there is a far greater choice of coordinates and corresponding basis sets, which are highly problem-dependent. What is essential for our analysis is that the naive unsymmetrized (but sparse) approximation scheme has an exponential rate of convergence. In order to maintain the generality of our results we now formulate this as an assumption, which, loosely speaking encapsulates that $\Phi$ is a ``good'' basis to approximate smooth functions  on $\Omega$ and therefore $\big[\Phi^{\otimes N}\big]$ is a good basis for approximation in $\Omega^N$:

\begin{assumption} \label{as:exprate}
  The target function $f \in C_{\rm sym}(\Omega^N)$ has a sparse polynomial approximate $f_D \in {\rm span}(\Phi^{\otimes N})$, satisfying 
  \begin{equation}
      \label{eq:ddim:totaldeg_apxerr}
      \| f - f_D \|_\infty \leq M e^{-\alpha' D},
  \end{equation}
  where $M, \alpha' > 0$,  $D$ is the total degree of $f_D$ defined by 
  \[
      D = \max_{\vii} {\rm deg}(\vii),
  \]
  and the maximum is taken over all basis functions $\phi_{\bm v}$ involved in the definition of $f_D$. 
  \end{assumption}

Arguing as above we have 
\[
  \| f - \operatorname{sym}[f_D] \|_\infty \leq 
  \| f - f_D \|_\infty, 
\]
hence we may again assume without loss of generality that $f_D$ is symmetric. It can be represented either in terms of the basis $\operatorname{sym}[\Phi^{\otimes N}]$ or equivalently in terms of the basis $\AA$, which optimizes the evaluation cost. This yields the representation 
\[
    f_D = \sum_{{\rm deg}(\vii) \leq D} \tilde{c}_{\vii} \AA_{\vii}.
\]
To obtain approximation rates in terms of the number of parameters,
\[
    P(N, D, d)  := \#\big\{ \vii \in \mathbb{N}^N \colon \text{ordered, and } {\rm deg}(\vii) \leq D \big\},
\]
we generalize the one-dimensional result of Hardy and Ramanujan~\cite{Hardy1918-gg}.

\begin{theorem}
\label{th:thm_estim_inf_N}
  Assume that Assumption~\ref{as:PHI} is satisfied, in particular {\rm($\Phi$4)} defining the intrinsic dimension $d$. Then, there exists a constant $\betad >0$ such that
  for any $N, D$,  
  \begin{equation} \label{eq:hardyramanujan_multiD}
      P(N, D, d) = \# \big\{ \vii \in \N^N_{\rm ord}: {\rm deg}(\vii) \le D \big\} 
      \leq 
      D e^{\betad D^{d/(d+1)}}.
  \end{equation}
\end{theorem}

\begin{proof}
  We adapt the proof of Erd\"os presented in~\cite{Erdos1942-qa} for the estimation of the number of integer partitions to the current context.
  First, we define the partition function
  \begin{equation}
      \label{eq:partition_function}
      F(x) = \prod_{i=1}^\infty (1 + g(i,1) x^i + g(i,2) x^{2i} + \cdots),
  \end{equation}
  where $g(i,k)$ denotes the number of $k$-body basis functions composed only of one-body basis functions with partial degree $i$.
  Then, in the decomposition of $F(x)$ as 
  \begin{equation}
      \label{eq:partition_function2}
      F(x) = \sum_{n\ge 0} f(n)x^n,
  \end{equation}
  the coefficient $f(n)$ is equal to the number of basis functions with total degree $n$. Note that there is no limitations on the number of terms in the product in~\eqref{eq:partition_function}, which corresponds to having no restriction on $N$.

  Now, denoting by $c_d(i)$ the number of one-body basis functions with degree $i$, the number $g(i,k)$ is the number of combinations with repetitions of $k$ in the set $\{1,2,\ldots, c_d(i) \}$, which is
  \[
      g(i,k) = {k + c_d(i)-1 \choose k} = {k + c_d(i)-1 \choose c_d(i)-1}.
  \]  
  Then, for $0< x <1$, a well-known identity on power series gives
      \begin{align*}
      \sum_{k=0}^\infty {k + c_d(i)-1 \choose c_d(i)-1} x^{k} = 
      \frac{1}{(c_d(i)-1)!} \; \frac{\mathrm{d}^{c_d(i)-1}}{\mathrm{d}x^{c_d(i)-1}} (1-x)^{-1} =
      \frac{1}{(1-x)^{c_d(i)}}.
  \end{align*}
   
  Hence, the partition function $F(x)$ can be written as
  \[
      F(x) = \prod_{i=1}^\infty \frac{1}{(1-x^i)^{c_d(i)}}.
  \]
  Taking the logarithm of $F(x)$, we obtain
  \[
      \log F(x) = - \sum_{i=1}^\infty c_d(i) \log (1-x^i),
  \]
  from which we deduce, differentiating this expression, that
  \[
      \frac{F'(x)}{F(x)} = \sum_{i=1}^\infty c_d(i) \frac{i x^{i-1}}{1-x^i}.
  \]
  Using the expansion~\eqref{eq:partition_function2}, we write
  $F'(x) = \sum_{n\ge 1} nf(n)x^{n-1}$, which leads to
  \[
      \sum_{n\ge 1} nf(n)x^{n-1} 
      = \sum_{l \ge 0} \sum_{i\ge 1} f(l) x^l c_d(i) \frac{i x^{i-1}}{1- x^i}.
  \]
  Therefore, expanding $(1-x^i)^{-1}$ as a power series,
  \[
      \sum_{n\ge 1} nf(n)x^{n-1} 
      = \sum_{l \ge 0} \sum_{i\ge 1} \sum_{k \ge 0} 
      f(l)  c_d(i) i x^{l+ i(k+1) -1}.
  \]
  Multiplying by $x$ on both sides, and substituting $k = k+1$, we obtain
  \[
      \sum_{n\ge 1} nf(n)x^{n} 
      = \sum_{l \ge 0} \sum_{i\ge 1} \sum_{k \ge 1} f(l) 
      c_d(i) i x^{l+ ik}.
  \]
  Hence, equating the coefficients of the power series on both sides, there holds
  \[
      n f(n) = \sum_{\substack{i\ge 1, k \ge 1 \\ ki \le n}}
      c_d(i) i  f(n - ki).
  \]
  Let us now show by induction that $f(n) \le e^{\betad n^{\alpha}}$, with
  \begin{equation}
      \label{eq:betad}
      \betad = \frac{d+1}{d} \left(c d! p_d \zeta(d+1) \right)^{1/(d+1)},
  \end{equation}
  where 
  \begin{equation}
      \label{eq:p_d}
      p_d = \sup_{x \in \R^+}  \frac{x^{d+1} e^{-x}}{(1-e^{-x})^{d+1}},
  \end{equation}
  $c$ comes from ($\Phi$4)
  and $\displaystyle\alpha = \frac{d}{d+1}$.
  
  First, $f(0) = 1$, therefore the property $f(n) \le e^{\betad n^{\alpha}}$ is satisfied for $n = 0$. 
   
  Now, we assume that for any $k < n$, $f(k) \le e^{\betad k^{\alpha}}$. Using the recurrence property, there holds
  \begin{align}
      \label{eq:proof_multi1}
      n f(n) &= \sum_{\substack{i,k \ge 1 \\ ki \le n}}
      c_d(i) i  f(n - ki) 
      \le  \sum_{\substack{i,k \ge 1 \\ ki \le n}}
      c_d(i) i e^{\betad (n-ki)^{\alpha}}.
  \end{align}
  Then, using the concavity of the function $x \mapsto x^\alpha$ and noting that $\frac{ki}{n} \le 1$, we obtain
  \begin{align*}
      (n-ki)^{\alpha} &= n^\alpha \left(1 - \frac{ki}{n} \right)^\alpha  \le n^\alpha \left(1 - \alpha \frac{ki}{n} \right).
  \end{align*}
  Inserting the latter bound in~\eqref{eq:proof_multi1} and using Assumption~($\Phi$4), we deduce
  \begin{align*}
      n f(n) 
      & \le
      c e^{\betad n^\alpha} \sum_{\substack{i,k \ge 1 \\ ki \le n}}
      (i+d-1) \dots (i+1) i   e^{- \betad \alpha ki n^{\alpha-1}}.
  \end{align*}
  Summing over all $i\in\N$, using that for $0< x <1$, 
  \begin{align*}
      \sum_{i\ge 1}^\infty (i+d-1) \dots (i+1) i x^{i} = 
      x \left[ \frac{1}{1-x}   \right]^{(d+1)} =
      \frac{d! \; x}{(1-x)^{d+1}},
  \end{align*}
  we obtain, taking $x = e^{- \betad \alpha k n^{\alpha-1}}$,
  \begin{align*}
      n f(n) 
      & \le
      c \; d! \; e^{\betad n^\alpha} 
      \sum_{k \ge 1} 
      \frac{e^{- \betad \alpha k n^{\alpha-1}}}{\left( 
         1- e^{- \betad \alpha k n^{\alpha-1}}
         \right)^{d+1} }
        .
  \end{align*}
  Introducing $p_d$ defined in~\eqref{eq:p_d}, we get
  \begin{align*}
      n f(n) &\le c \; d! \; p_d \; e^{\betad n^\alpha} \sum_{ k \ge 1}
      \frac{1}{ ( \betad \alpha k n^{\alpha-1})^{d+1}} \\
      & =  c \; d! \; p_d \;  \frac{\zeta(d+1) }{\betad^{d+1} \alpha^{d+1}}  n^{(d+1)(1-\alpha)} e^{\betad n^\alpha}.
  \end{align*}
  Since $\alpha = \frac{d}{d+1}$ and $\betad = \frac{d+1}{d} \left(c d! p_d \zeta(d+1) \right)^{1/(d+1)}$, we obtain
  \[
      n f(n) \le n e^{\betad n^\alpha},
  \]
  from which we deduce that for all $n\in\N$, $f(n) \le e^{\betad n^\alpha}$, i.e.
  \[
      \# \big\{ \vii \in \N^N_{\rm ord}: {\rm deg}(\vii) = D \big\} \le e^{\betad D^{d/(d+1)}}.
  \]
  We then easily deduce the result of the theorem.
\end{proof}

With the basis size estimate in hand, we can now readily extend the one-dimensional approximation rate.

\begin{theorem} \label{th:general:main}
  Under Assumptions~\ref{as:PHI} and \ref{as:exprate} there exist constants $\alpha, M > 0$ such that 
  \[
      \| f - f_D \|_\infty 
      \leq M  
      e^{ - \alpha [\log P]^{1+1/d} }
  \]
\end{theorem}
\begin{proof}
  This follows immediately from \eqref{eq:ddim:totaldeg_apxerr} and, 
  denoting $c_d = \max_D \frac{1}{\beta_d} \frac{\log D}{D^{d/(d+1)}}$, using that $\log P \leq \beta_d D^{d/(d+1)} + \log D$ implies\[
     - D \leq - \left[\frac{\log P}{(1+c_d)\beta_d} \right]^{1+1/d}
  \]
  and absorbing $\beta_d$ into $\alpha := \alpha' / ((1+c_d)\beta_d)^{1+1/d}$.
\end{proof}

\subsection{Asymptotic Results for $D \gg N^{1+1/d}$}
\label{sec:sharpness}
Our results up to this point are uniform in $N, D$ but in fact, they turn out to be sharp only in the regime of fairly moderate degree $D$, specifically for $D \lesssim N^{1+1/d}$. Our final set of results provides improved complexity and error estimates for the regime $D \gg N^{1+1/d}$. In particular, they will also provide strong indication that Theorem~\ref{th:general:main} is sharp in the regime $D \lesssim N^{1 + 1/d}$.

Here, and below it will be convenient to use the symbols $\lesssim, \gtrsim, \eqsim$ to mean less than, greater than, or bounded above and below up to uniform positive constants.

Our subsequent analysis is based on the fact that there exist constants $c_0, {c_1}$ such that, for $D \geq c N^{1+1/d}$ with $c$ sufficiently large, 
\begin{equation} \label{eq:general:assume_bounds}
  c_0^N \frac{D^{dN}}{(dN)! N!} 
  \leq P 
  \leq 
  c_1^N \frac{D^{dN+1}}{(dN)! N!}.
\end{equation}
The lower bound is straightforward to establish. The upper bound for $d = 1$ follows immediately from \eqref{eq:params_regimes}. For $d > 1$, and in the limit $D \to \infty$ it is again straightforward to establish. We prove a slightly modified generic upper bound for $d \ge 1$ below. 
We then recover the upper bound \eqref{eq:general:assume_bounds} from the following theorem noting that $\frac{D^{dN + 1}}{(dN)! (N-1)!} \leq c^N  \frac{D^{dN+1}}{(dN)! N!}$ with $c = 3^{1/3} \approx 1.44$.
(The constants in (\ref{eq:general:assume_bounds}) and  (\ref{thm:hardyramanujan_multiD}) are different). 

\begin{theorem}
   Assume that Assumption~\ref{as:PHI} is satisfied, in particular {\rm($\Phi$4)} defining the intrinsic dimension $d$. Then, there exist two constants $c_1,c_2>0$ that depend on $d$ such that
   for any $N, D$,  
   \begin{equation} \label{thm:hardyramanujan_multiD}
      P(N, D, d) = \# \big\{ \vii \in \N^N_{\rm ord}: {\rm deg}(\vii) \le D \big\} 
      \leq 
      D c_1^N \frac{(D+c_2 N^{\frac{d+1}{d}})^{dN}}{ (dN)! (N-1)!}.
   \end{equation}
\end{theorem}

\begin{proof}
   First, we define the partition function
   \begin{equation}
      \label{eq:partition_function_xy}
      F(x,y) = \prod_{i=1}^\infty (1 + g(i,1) x^i y + g(i,2) x^{2i} y^2 + \cdots),
   \end{equation}
   where $g(i,k)$ denotes the number of $k$-body basis functions composed only of one-body basis function with partial degree $i$.
   Then, in the decomposition of $F(x,y)$ as 
   \begin{equation}
      \label{eq:partition_function2_xy}
      F(x,y) = \sum_{n,m\ge 0} f(n,m)x^n y^m,
   \end{equation}
   the coefficient $f(n,m)$ is equal to the number of basis functions with total degree $n$ and number of parts $m$. 

   Now, as in the proof of Theorem~\ref{th:thm_estim_inf_N},
   denoting by $c_d(i)$ the number of one-body basis functions of degree $i$, 
   we can write the partition function $F(x,y)$ as
   \[
      F(x,y) = \prod_{i=1}^\infty \frac{1}{(1-x^i y)^{c_d(i)}}.
   \]
   Taking the logarithm of $F(x,y)$, we obtain
   \[
      \log F(x,y) = - \sum_{i=1}^\infty c_d(i) \log (1-x^i y),
   \]
   from which we deduce, differentiating this expression with respect to $y$, that
   \[
      \frac{\frac{\partial F}{\partial y}(x,y)}{F(x,y)} = \sum_{i=1}^\infty c_d(i) \frac{ x^i}{1-x^i y}.
   \]
   Using the expansion~\eqref{eq:partition_function2_xy}, we write
   $\frac{\partial F}{\partial y}(x,y) = \sum_{n,m\ge 0} m f(n,m) x^{n} y^{m-1}$, which leads to
   \[
   \begin{aligned}
      \sum_{n,m\ge 0} m f(n,m) x^{n} y^{m-1}
      &= \sum_{p,q \ge 0} f(p,q) x^p y^q \sum_{i\ge 1}  c_d(i) x^i \left[
      \sum_{k\ge 0} (x^i y)^k \right] \\
      &=\sum_{p,q \ge 0} \sum_{i\ge 1}  \sum_{k\ge 0} f(p,q) c_d(i) x^{p+i(k+1)} y^{k+q}.
   \end{aligned}
   \]
   Multiplying by $y$ on both sides, and substituting $k+1$ for $k$, we obtain
   \[
      \sum_{n,m\ge 0} m f(n,m) x^{n} y^{m}
      =\sum_{p,q \ge 0} \sum_{i,k\ge 1} f(p,q) c_d(i) x^{p+ik} y^{k+q}.
   \]
   Hence, equating the coefficients of the power series on both sides, we conclude
   \begin{equation}
   \label{eq:rec_equation}
             m f(n,m) = \sum_{\substack{i\ge 1, k \ge 1 \\ ki \le n \\ k \le m}}
      c_d(i) f(n - ki,m-k).
   \end{equation}
We now show by induction that for any $c_1,c_2>0$ that satisfy
\begin{equation}\label{eq:c1c2cond}
  c_1 c_2^d > d^d \;\text{ and }\;  (c_1-\widetilde{c} \alpha_d) c_2^d \ge d^d,
\end{equation}
where $\tilde c$ is the constant appearing in ($\Phi$4), $\alpha_1 = 1$, and $\alpha_d = \max\{(d-1)!\, 2^d, (d-1)! + d(d-1)^{d-1}\}$ for $d\ge 2$, we have
\begin{equation}\label{eq:festimate}
 f(n,m) \le c_1^m \frac{(n+c_2 m^{\frac{d+1}{d}})^{dm}}{(dm)! m!}.
  \end{equation}
Note that \eqref{eq:c1c2cond} holds in particular for the choice $c_1 = 1+ \widetilde{c} \alpha_d$ and $c_2 = d$.
  
   First, we observe that $f(0,0) = 1$, $f(n,0) = 0$ for $n\ge 1$, $f(0,m) = 0$ for $m\ge 1$, $f(1,1) = c_d(1) \le \widetilde{c}$, and $f(1,m) = 0$ for all $m\ge 2$. Thus \eqref{eq:festimate} holds in each of these cases. 

 Now for arbitrary but fixed $m$ and $n$, we assume that \eqref{eq:festimate} holds true for any $k < m$, $i < n$. Using the recurrence property, starting from~\eqref{eq:rec_equation}, and noting that $c_d(i)$ appearing in ($\Phi$4) can be bounded by $\tilde c i^{d-1}$ for some $\tilde c>0$, we find
   \begin{align}
       m f(n,m) 
          & \le  \sum_{\substack{i,k \ge 1 \\ ki \le n \\ k \le m}}
      c_d(i) c_1^{m-k} \frac{\left(n-ik + c_2(m-k)^{\frac{d+1}{d}}\right)^{d(m-k)}}{(d(m-k))! (m-k)!} \nonumber \\
      & \le \widetilde{c}
      \sum_{1\le k\le m} \frac{c_1^{m-k}}{(d(m-k))! (m-k)!}
      S(n,k,c_2 (m-k)^{\frac{d+1}{d}}, d(m-k), d-1),
      \label{eq:estim_with_S}
   \end{align}
   with
   \[
        S(n,k,\gamma,p,q) = \sum_{1\le i \le \floor*{\frac{n}{k}}} (n-ik + \gamma)^p i^q,
   \]
   where $\gamma \geq 0$, $p,q,k \in \N$, $k\ge 1$. For estimating this quantity, we use the following two lemmas.
   
   \begin{lemma}
    \label{lem:S_estim1}
      For $\gamma \geq 0$, $n,k,p,q  \in \N$, $k\ge 1$, we have
      \begin{equation}\label{eq:Sbound}
        S(n,k,\gamma,p,q) \le \beta(q,p,k,\gamma) \frac{(n+\gamma)^{p+q+1}}{(p+1)(p+2) \ldots (p+q+1)},
      \end{equation}
      with $\beta(0,p,k,\gamma)=1$, $\beta(q,0,k,\gamma) = q!\, 2^{q+1}$, and $\beta(q,p,k,\gamma) = \frac{q! }{k^{q+1}} + \frac{p^p q^q }{k^q (1+\gamma)}
              \frac{(p+q+1)}{(p+q)^p}$ for $p,q\ge 1$. 
   \end{lemma}
   \begin{proof}
    To prove \eqref{eq:Sbound}, we introduce the function $g: x \mapsto (n-kx+\gamma)^p x^q$ and compare the sum $S(n,k,\gamma,p,q) = \sum_{1\le i \le \floor*{\frac{n}{k}}} g(i)$ to the integral of $g$.
    Note that \eqref{eq:Sbound} trivially holds for $n = 0$. Hence we assume $n\ge 1$.
    We treat three different cases depending on the values of $p$ and $q$.
   
    If $q=0$, the function $g$ decreases on $[0, \frac{n}{k}]$. Therefore, 
      \[
        S(n,k,\gamma,p,q) \le \int_0^{\frac{n}{k}} g(x) dx = \left[ -\frac{(n-kx+\gamma)^{p+1}}{(p+1)k} \right]_0^{\frac{n}{k}} \le 
        \frac{(n+\gamma)^{p+1}}{(p+1)k}.
      \]
      Since $k \ge 1$, \eqref{eq:Sbound} holds in this case, with $\beta(0,p,k,\gamma) = 1$.
      
     If $p=0$, using $k\ge 1$, we obtain
      \[
        S(n,k,\gamma,p,q) \le 
        \frac{(\frac{n}{k}+1)^{q+1}}{q+1} \le \frac{(n+1)^{q+1}}{q+1} \le \frac{n^{q+1}}{(q+1)!} q! \left(\frac{n+1}{n}\right)^{q+1}.
      \]
      Hence \eqref{eq:Sbound} is satisfied with $\beta(q,0,k,\gamma) = q! \,2^{q+1}$.
      
      As the third case, we consider $p \ge 1, q \ge 1$. For $x \ge 0$,
      \[
        g'(x) = (n-kx + \gamma)^{p-1} x^{q-1} \left( q(n+\gamma) - xk(p+q) \right),
      \]
      so the function $g$ is increasing on $\left[0, \frac{q(n+\gamma)}{k(p+q)}\right]$, and decreasing on $\left[\frac{q(n+\gamma)}{k(p+q)} ,+\infty\right[$. In particular, the function has a single local maximum in $x = \frac{q(n+\gamma)}{k(p+q)}$. Hence, there holds
      \[
        \sum_{1\le i \le \floor*{\frac{n}{k}}} g(i) \le 
        \int_0^{\frac{n}{k}} g(x) dx + g\left(\frac{q(n+\gamma)}{k(p+q)}\right).
      \]
      Integrating by parts $q$ times, we obtain
      \begin{align*}
          \int_0^{\frac{n}{k}} g(x) dx &= \left[ 
          - \frac{x^q (n-kx+\gamma)^{p+1}}{k(p+1)} \right]_0^{\frac{n}{k}} +  \int_0^{\frac{n}{k}}
          \frac{q}{k(p+1)} (n-kx+\gamma)^{p+1} x^{q-1} dx \\
          & \le \frac{q}{k(p+1)} \int_0^{\frac{n}{k}}
          (n-kx+\gamma)^{p+1} x^{q-1} dx \\
          & \le \frac{q!}{k^q(p+1) \ldots (p+q)}  \int_0^{\frac{n}{k}}
          (n-kx+\gamma)^{p+q} dx \\
          & \le \frac{q! (n+\gamma)^{p+q+1}}{k^{q+1}(p+1) \ldots (p+q+1)} .
      \end{align*}
       Moreover,
      \[
        g \left( \frac{q(n+\gamma)}{k(p+q)} \right) = \left(n-\frac{q(n+\gamma)}{p+q} +\gamma \right)^p
        \left( \frac{q(n+\gamma)}{k(p+q)} \right)^q 
        = \left(n+\gamma \right)^{p+q}
         \frac{p^p q^q}{k^q(p+q)^{p+q}}.
      \]
      Since $n,k,p,q \ge 1,$ we easily deduce
      \begin{align*}
              g \left( \frac{q(n+\gamma)}{k(p+q)} \right) 
              & = \frac{(n+\gamma)^{p+q+1}}{{(p+1)\ldots (p+q+1)}} 
              \frac{p^p q^q }{k^q (n+\gamma)}
              \frac{{(p+1)\ldots (p+q) (p+q+1)}}{(p+q)^{p+q}}
              \\
              & \le \frac{(n+\gamma)^{p+q+1}}{{(p+1)\ldots (p+q+1)}} 
              \frac{p^p q^q }{k^q (n+\gamma)}
              \frac{(p+q+1)}{(p+q)^p},
      \end{align*}
     and thus
      \[
         S(n,k,\gamma,p,q) \le \frac{(n+\gamma)^{p+q+1}}{ (p+1) \ldots (p+q+1)} 
         \left(\frac{q! }{k^{q+1}} + \frac{p^p q^q }{k^q (n+\gamma)}
              \frac{(p+q+1)}{(p+q)^p} \right).
      \]
      Therefore, \eqref{eq:Sbound} holds also in this third case, with $\beta(q,p,k,\gamma) = \frac{q! }{k^{q+1}} + \frac{p^p q^q }{k^q (1+\gamma)}
              \frac{(p+q+1)}{(p+q)^p}$.
     This concludes the proof of the lemma.
   \end{proof}
   
   We now show that the prefactor can be bounded uniformly in $m$ and $k$.
   \begin{lemma}
      For $m,k\in\N,$ $m,k\ge 1$, we have
      \begin{equation}\label{eq:betabound}
           \beta\bigl(d-1,d(m-k),k,c_2 (m-k)^{\frac{d+1}{d}}\bigr) \le \alpha_{d},
      \end{equation}
      with $\alpha_1 = 1,$ and $\alpha_d = \max\{ (d-1)!\, 2^d, (d-1)! + d(d-1)^{d-1}\}$ for $d\ge 2$, provided that $c_2 \ge 1$.
   \end{lemma}

   \begin{proof}
      If $d=1$ (corresponding to $q=0$ in the previous lemma) or $k = m$ (corresponding to $p=0$), the estimate \eqref{eq:betabound} holds with $\alpha_1 = 1$ and $\alpha_d = (d-1)!\, 2^d$, respectively, by Lemma~\ref{lem:S_estim1}.
      Now if $1\le k < m$, 
      \begin{multline*}
       \beta\bigl(d-1,d(m-k),k,c_2 (m-k)^{\frac{d+1}{d}}\bigr) \\ =
          \frac{(d-1)!}{k^d} + \frac{[d(m-k)]^{d(m-k)} (d-1)^{(d-1)} }{k^{d-1} (1+c_2 (m-k)^{\frac{d+1}{d}})}
              \frac{(d(m-k) + d)}{(d(m-k)+d-1)^{d(m-k)}} 
              \\
               \le (d-1)! + \frac{(d-1)^{(d-1)}}{k^{d-1}}
              \frac{d (m-k+1)}{(1+c_2 (m-k)^{\frac{d+1}{d}})}
              \left[\frac{d(m-k)}{d(m-k)+d-1} \right]^{d(m-k)}.
      \end{multline*}
      Noting that $\frac{d(m-k)}{d(m-k)+d-1} \le 1$, $k\ge 1$, and that 
      $\frac{d (m-k+1)}{1+c_2 (m-k)^{\frac{d+1}{d}}} \le d$ if $c_2 \ge 1$,
      we obtain
      \[
        \beta\bigl( d-1,d(m-k),k,c_2 (m-k)^{\frac{d+1}{d}} \bigr)
        \le (d-1)! + d(d-1)^{d-1},
      \]
      completing the proof of \eqref{eq:betabound}.
   \end{proof}

  Coming back to~\eqref{eq:estim_with_S} and using the estimates of the two lemmas, we obtain
    \begin{align*}
        m f(n,m) &\le \widetilde{c} \alpha_d
      \sum_{1\le k\le m} \frac{c_1^{m-k}}{(d(m-k))! (m-k)!}
      \frac{(n+c_2(m-k)^{\frac{d+1}{d}})^{d(m-k)+d}}{(d(m-k)+1)(d(m-k)+2) \ldots (d(m-k)+d)} \\
      &=
      \widetilde{c} \alpha_d
      \sum_{1\le k\le m} \frac{c_1^{m-k} (n+c_2(m-k)^{\frac{d+1}{d}})^{d(m-k+1)}}{(d(m-k+1))! (m-k)!} \\
      &\le c_1^m \frac{(n+c_2 m^{\frac{d+1}{d}})^{dm}}{(dm)! m!}
      \widetilde{c} \alpha_d 
      \sum_{1\le k\le m} \frac{1}{c_1^k}
      \frac{(dm)!}{(d(m-k+1))!} \frac{m!}{(m-k)!}
      \frac{(n+c_2(m-k)^{\frac{d+1}{d}})^{d(m-k+1)}}{(n+c_2 m^{\frac{d+1}{d}})^{dm}}.
    \end{align*}
    Combining this with the estimates
    \[
    \begin{gathered} \frac{(dm)!}{(d(m-k+1))!} \le (dm)^{d(k-1)}, \qquad \frac{m!}{(m-k)!} \le m^k,  \\ \frac{(n+c_2(m-k)^{\frac{d+1}{d}})^{d(m-k+1)}}{(n+c_2 m^{\frac{d+1}{d}})^{dm}} \leq \frac{1}{(n + c_2m^{\frac{d+1}{d}})^{d(k-1)}} \leq \frac{1}{(c_2m^{\frac{d+1}{d}})^{d(k-1)}}
    \end{gathered}
    \]
    yields
    \begin{equation*}
         m f(n,m)  \le 
      c_1^m \frac{(n+c_2 m^{\frac{d+1}{d}})^{dm}}{(dm)! m!}
      \widetilde{c} \alpha_d 
      \sum_{1\le k\le m} \frac{(dm)^{d(k-1)} m^k}{c_1^k (c_2m^{\frac{d+1}{d}})^{d(k-1)}}.
    \end{equation*}
    Noting that $\displaystyle  \frac{(dm)^{d(k-1)} m^k}{(c_2m^{\frac{d+1}{d}})^{d(k-1)}} =  m \frac{d^{d(k-1)}}{c_2^{d(k-1)}}$ we obtain, replacing the index $k$ with $k-1$,
    \begin{align*}
         m f(n,m) 
      & \le c_1^m \frac{(n+c_2 m^{\frac{d+1}{d}})^{dm}}{(dm)! m!}
      \widetilde{c} \frac{\alpha_d m}{c_1}
      \sum_{0\le k\le (m-1)} \left( \frac{d^d}{c_1 c_2^d} \right)^k.
    \end{align*}
    Taking $c_1,c_2$ such that $c_1 c_2^d > d^d$ and dividing by $m$ yields
    \[
        f(n,m) \le c_1^m \frac{(n+c_2 m^{\frac{d+1}{d}})^{dm}}{(dm)! m!}
      \widetilde{c} \frac{\alpha_d}{c_1} \frac{1}{1- \frac{d^d}{c_1 c_2^d}}.
    \]
    Therefore, \eqref{eq:festimate} is satisfied if 
    $ \displaystyle
        \widetilde{c} \alpha_d c_2^d \le c_1 c_2^d- d^d,
    $
    that is, if
    $\displaystyle
        c_1 c_2^d \ge \widetilde{c} \alpha_d c_2^d + d^d.
    $
    This completes the induction. 
    Hence, we deduce that
    \begin{align*}
        \# \big\{ \vii \in \N^N_{\rm ord}: {\rm deg}(\vii) = D \big\} & = \sum_{m=0}^N f(D,m)  \le 
      \sum_{m=0}^N c_1^m \frac{(D+c_2 m^{\frac{d+1}{d}})^{dm}}{(dm)! m!}.
    \end{align*}
      Finally, we note that a sufficient condition for $\N \ni m \mapsto c_1^m \frac{(D+c_2 m^{\frac{d+1}{d}})^{dm}}{(dm)! m!}$ to be increasing is that $c_1 c_2^d \ge d^d$. Since this is satisfied under by \eqref{eq:c1c2cond},
      \begin{equation*} 
      \# \big\{ \vii \in \N^N_{\rm ord}: {\rm deg}(\vii) = D \big\} 
      \le 
      c_1^N \frac{(D+c_2 N^{\frac{d+1}{d}})^{dN}}{(dN)! (N-1)!}.
    \end{equation*}
    Summing over the possible values of $D$, we 
    easily deduce the result of the Theorem.
\end{proof}

With this result at hand, we now consider the dependency of the number of parameters on $D$.

\begin{lemma} 
  \label{th:large_D_paramsest}
  Suppose that $D = N^t$ with $t \geq 1 + 1/d$. Then, for sufficiently large $D$,
  \begin{equation}
      \log P 
      \eqsim D^{1/t} + (t -1 - 1/d) D^{1/t} \log D^{1/t}, 
  \end{equation}
  uniformly for all $t \geq 1+1/d$ (but the constants may depend on $d$).
\end{lemma}

\begin{proof}
  From \eqref{eq:general:assume_bounds} we can estimate 
  \begin{align*}
      \log P 
      &\leq 
      N \log c_1 +  dN \log D - d N \log\big(dN / e\big) - N \log\big(N/e\big) + \log D
      \\ 
      &= 
      dN \Big( c_2 + \log D - \log N - \log N^{1/d} \Big) + \log D
      \\ 
      &= 
      d D^{1/t} \Big( c_2 + \log D^{1 - 1/t - 1/dt}) + \log D.      
  \end{align*}
  Noting that asymptotically as $D \to \infty$, $\log D \ll D^{1/t}$ we easily obtain the upper bound. The lower bound is analogous. 
\end{proof}

\begin{proposition} \label{th:general:improvedest}
  Suppose that $D = N^t$ for some $t > 1 + 1/d$, and that $N, D$ are sufficiently large, then 
  \begin{equation} \label{eq:general:improvedest1}
      D 
      \gtrsim
      \frac{1}{(t - 1 - 1/d)^t}
      \bigg[ \frac{\log P}{\log\log P} \bigg]^t.
  \end{equation}
\end{proposition}
\begin{proof}
  From Lemma~\ref{th:large_D_paramsest}, for the case $t > 1 + 1/d$ and $N, D$ sufficiently large we have 
  \begin{equation} \label{eq:prflargeD:100}
      a_0 (t - 1 - 1/d) D^{1/t} \log D^{1/t} \leq \log P \leq a_1 (t - 1 - 1/d) D^{1/t} \log D^{1/t}. 
  \end{equation}
  From the lower bound we deduce 
  \[
      \log (a_0 (t - 1 - 1/d)/t) + \log D^{1/t} + \log\log D^{1/t} \leq
      \log\log P,
  \]
  and for $D$ sufficiently large, this implies 
  \[
      \log D^{1/t} \leq \log\log P, \qquad \text{or,} \qquad 
      \frac{1}{\log\log P} 
      \leq \frac{1}{\log D^{1/t}}.
  \]
  Multiplying this with the right-hand inequality in \eqref{eq:prflargeD:100} yields 
  \[
      \frac{\log P}{\log\log P} 
      \leq 
      a_1 (t - 1 - 1/d) D^{1/t},
  \]
  and taking this to power $t$ gives the stated result.
\end{proof}

In particular, in the regime $D = N^t, t > 1+1/d$, if we start from the exponential convergence rate of Assumption~\ref{as:exprate}, then we obtain that there exists a symmetric approximation $f_D$ with 
\begin{equation} \label{eq:apxrate_largeD}
  \|f - f_D\|_\infty \lesssim \exp\bigg( - \alpha (t-1-1/d)^{-t}  \bigg[ \frac{\log P}{\log\log P} \bigg]^t \bigg),
\end{equation}
for some $\alpha > 0$. This rate is clearly superior to the uniform $N$-independent rate of Theorem~\ref{th:general:main} (in this regime at least). Moreover, as $t \to 1 + 1/d$ it strongly hints at a kind of ``phase transition'', in other words, that the behaviour of the approximation scheme significantly changes at that point. 

\begin{remark} 
  It is furthermore interesting to compare our result \eqref{eq:apxrate_largeD} with analogous estimates if we had not exploited symmetry, or sparsity. An analogous analysis shows that if we use a (sparse) total degree approximation, but use a naive basis instead of a symmetrized basis, then we obtain 
  \[
      \|f - f_D\|_\infty \lesssim \exp\bigg( - \alpha_2 (t-1)^{-t}  \bigg[ \frac{\log P}{\log\log P} \bigg]^t \bigg),
  \]
  Finally, if we drop even the sparsity, so that $P \approx D^{dN}$, then we obtain 
  \[
      \|f - f_D\|_\infty \lesssim
      \exp\bigg( - \alpha_3 t^{-t}\bigg[ \frac{\log P}{\log\log P} \bigg]^t \bigg)
  \]
  Note here that the $t$ dependence of the exponents highlights quite different qualitative behaviour of the three bases. 
  Specifically, this provides strong qualitative evidence that (unsurprisingly) the sparse basis gives a significant advantage over the naive tensor product basis especially in the regime $1 \leq t$, but that it is still significantly more costly than the sparse symmetrized basis. Moreover, these estimates clearly highlight the unique advantage that the symmetrized basis provides in the regime $D \lesssim N^{1+1/d}$ treated in Theorem~\ref{th:general:main}.
\end{remark} 

\newcommand{\jtedit}[1]{\textcolor{magenta}{#1}}

\section{Approximation of multi-set functions}
\label{sec:mset}
Our original motivation to study the approximation of symmetric functions arises in the atomic cluster expansion~\cite{Drautz2019-rb,Bachmayr2019-ec} which is in fact concerned with the approximation of multi-set functions. We now study how our approximation results of the previous sections can be applied in this context. 
Similarly as in the previous sections, this discussion will again ignore isometry-invariance.
We will focus on a general abstract setting, but employ assumptions that can be rigorously established in the setting of~\cite{Drautz2019-rb,Bachmayr2019-ec}. 

Let $\ms(\Omega)$ denote the set of all multi-sets (or, msets) on $\Omega$; i.e., 
\begin{equation} \label{eq:msets}
    \ms(\Omega) := 
    \big\{ {\bm x} = [x_1,\dots, x_M] \colon M \in \mathbb N, \{x_j\}_{j=1}^M \subset \Omega \big\}, 
\end{equation}
where $[x_1, \dots, x_M]$ denotes an {\em unordered tuple} or mset, e.g. describing a collection of (positions of) $M$ classical particles. The crucial aspect is that $M$ is now variable and no longer fixed. 
This is equivalent to the classical definition of a multiset which has the defining feature of allowing for multiple instances for each
of its elements.
We are interested in parameterising (approximating) mset functions
\begin{gather}\label{eq:f}
    f \colon \ms(\Omega) \to \mathbb R.
\end{gather}

\begin{remark}[Context] \label{rem:ms:context}
    This situation arises, e.g., when modelling interactions between particles. Different local or global structures of particle systems lead to a flexible number of particles entering the range of the interaction law. It seems tempting to take the limit $M \to \infty$, which leads to a mean-field-like scenario where a signal processing perspective could be of interest. However, we are interested in an intermediate situation where $M$ is ``moderate''; say, in the range $M = 10$ to $100$, and thus this limit is not of interest. Indeed, the number of interacting particles $M$ can be understood as another approximation parameter, which we discuss in more detail in Remark~\ref{rem:cut-off} below.
\end{remark}

In order to reduce the approximation of an mset function $f$ to our foregoing results we must first produce a representation of $f$ in terms of finite-dimensional symmetric components. A classical idea is the many-body or ANOVA expansion, which we will formulate as an {\em assumption}. However, we emphasize that our recent results \cite{ThomasChenOrtner2021:body-order} rigorously justify this assumption in the context of coarse-graining electronic structure models into interatomic potentials models. The following formulation is modelled on those results, which are summarized in Appendix~\ref{app:body-order-assumption}.

\begin{assumption} \label{as:body-order}
    {\it (i)} For all 
    $N$, there exist symmetric $V_{nN}\colon \Omega^n \to \mathbb R$ for $n = 0,\dots, N$, and $\eta > 0$ such that $f_N \colon \ms(\Omega) \to \mathbb R$, defined by
\begin{gather}
    f_N([x_1,\dots,x_M]) := V_{0N} + \sum_j V_{1N}(x_j) + \dots + \sum_{j_1 < \dots < j_N} V_{NN}(x_{j_1}, \dots, x_{j_N}) \nonumber \\
    \text{satisfies} \qquad 
    |f(\bm x) - f_N(\bm x)| \lesssim e^{-\eta N}\label{body-order-approx}
\end{gather}
    for all $\bm x \in \mathrm{MS}(\Omega)$.
    
    {\it (ii)} Moreover, we suppose that each $V_{nN}$ has a sparse polynomial approximation, $V_{nND}$ where $D$ denotes the total degree, satisfying 
    \begin{equation} \label{eq:VnND-estimate}
            \big|V_{nN}(x_{1}, \dots, x_{n})- V_{nND}(x_{1}, \dots, x_{n})\big| \lesssim c^n e^{-\alpha_n D}
            \qquad \forall x_1,\dots,x_n \in \Omega
    \end{equation} 
    for some $c, \alpha_n > 0$ and $n = 1, \dots, N$. Here, $V_{0N}$ is simply a constant term and requires no approximation.
\end{assumption}

Part (i) of Assumption~\ref{as:body-order} is the main result of \cite{ThomasChenOrtner2021:body-order}, while part (ii) essentially encodes the assumption that the $V_{nN}$ are analytic, with the region of analyticity encoded in $\alpha_n$ and possibly varying. Under a suitable choice of one-particle basis, one then obtains (\ref{eq:VnND-estimate}). Note in particular, that in this context the dimensionality of the target functions $V_{nN}$ is an approximation parameter, which makes it particularly natural to consider the different regimes how $D$ and $N$ are related in the foregoing sections.

Throughout the following discussion suppose that Assumption~\ref{as:body-order} is satisfied. 
Then, for a tuple ${\bm D} = (D_1,\dots, D_N)$ 
specifying the total degrees $D_n$ used to approximate the components $V_{nN}$, the cluster expansion approximation to $f$ is given by 
\begin{equation} \label{eq:defn_cluster_expansion}
    f_{N{\bm D}}(\bm x) := V_{0N} + \sum_{n=1}^N \sum_{j_1 < \dots < j_n} V_{nND_n}(x_{j_1},\dots,x_{j_n}),
\end{equation}
with $V_{nND_n}$ of the form 
    $V_{nND_n}(x_{j_1},\dots,x_{j_n})
    =  \sum_{v_1, \dots, v_n} c^{nND_n}_{\bm v} \prod_{t = 1}^n \phi_{v_t}(x_{j_t})$, for some coefficients $c^{nND_n}_{\bm v}$ and one-body basis functions $(\phi_{v})_{v\in \N}$. 
We immediately deduce an approximation error estimate: for all $\bm x = [x_1,\dots,x_M] \in \ms(\Omega)$, 
\begin{align}
    \left| f(\bm x) -  f_{N\bm D}(\bm x)\right| 
    &\leq |f(\bm x) - f_N(\bm x)| 
    + \sum_{n=1}^N \sum_{1 \leq j_1 < \dots < j_n \leq M} \big| (V_{nN} - V_{nND})(x_{j_1},\dots,x_{j_n})\big|\nonumber\\
    \label{eq:error_f_fND}
    &\lesssim e^{-\eta N} + \sum_{n=1}^N  {M \choose n} c^n e^{-\alpha_n D_n}.
\end{align}

From this expression we can now minimize the computational cost subject to the constraint that the overall error is no worse than the many-body approximation error $e^{-\eta N}$ and then obtain a resulting error vs cost estimate.

\begin{remark}
     \label{rem:cut-off}
    As we already hinted in Remark~\ref{rem:ms:context}, the function $f$ is often more naturally defined on the whole of $\ms(\mathbb R^d)$ (or even on the space of infinite multi-sets) but only ``weakly'' depends on points far away. That is, on defining $\bm x_R := [ x \in \bm x \colon |x| \leq R ]$ (i.e.~we remove from $\bm x$ only points outside $B_R$) we have
    \[
        \big| f(\bm x) - f(\bm x_R) \big| \lesssim e^{-\gamma R},
    \]
    that is, the domain $\Omega$ itself becomes an approximation parameter as well. Such exponential ``locality'' results arise for example in modelling of interatomic interaction laws~\cite{ThomasChenOrtner2021:body-order,Chen2016-ig}. For the sake of simplicity we will not incorporate this feature into our analysis, except for a natural assumption on how $M$ and $N$ are related: 
    
    If we approximate the restriction of $f$ to $\ms(B_R)$ then we obtain
    \begin{align}
        |f(\bm x) - f_{N}(\bm x_R)| 
        \lesssim e^{-\gamma R} + e^{-\eta N}
    \end{align}
       Balancing the error, we choose $\gamma R = \eta N$. In many physical situations we can assume that particles do not cluster and this leads to the bound $M \lesssim R^d$. More generally, we will therefore assume below that $M$ is bounded by a polynomial in $N$, which will make our analysis a little more concrete.
\end{remark}

\begin{remark}[Connection to Deep sets]
    \label{rem:deepsets}
    Deep set architectures are based on the idea that set-functions $f \colon \bigcup_{m=0}^M \big(\mathbb R^d\big)^m \to \mathbb R$ which are continuous when restricted to any fixed number of inputs are symmetric if and only if there exist 
    continuous functions $\phi \colon \mathbb R^d \to \mathbb R^Z$ and $\rho \colon \mathbb R^Z \to \mathbb R$ such that  
    \begin{align}\label{eq:deepsets}
        f(\bm x) = \rho\Big( \textstyle\sum_{x\in \bm x} \phi(x) \Big).  
    \end{align}
    For $d=1$, it has been shown that $Z = M$ is sufficient \cite{Zaheer2017:deepsets}, and there exist functions where $Z = M$ is also necessary \cite{Wagstaff2019}. Motivated by this characterisation, one obtains a deep set approximation to $f$ by choosing a latent space dimension $Z$ and learning the mappings $\phi$ and $\rho$, e.g. parametrised using multi-layer perceptrons.
    
    The cluster expansion approximation $f_{N\bm D}$ (\ref{eq:defn_cluster_expansion}) can therefore be thought of as a special case of (\ref{eq:deepsets}) with both $\phi$ and $\rho$ polynomials. However, the conceptual key difference in our setting is that, on the one hand, $Z$ is significantly larger than $d$ and should be taken as an approximation parameter, while on the other hand, the embedding $\phi$ is known {\it a priori} and need not be learned. Despite these philosophical differences, one can see the approximation results of this section as approximation rates for deep set approximations.
\end{remark}

\subsection{Computational cost of the cluster expansion}
The expression in \eqref{eq:defn_cluster_expansion} suggests that the evaluation cost scales as ${M \choose N}$ (notwithstanding the cost of evaluating the $V_{nN}$ components), but in fact a similar transformation as in \S~\ref{sec:1d} allows us to reduce this to a cost that is linear in $M$; see also \cite{Drautz2019-rb,Bachmayr2019-ec}.

We consider a single term, 
\begin{align*}
    \sum_{j_1 < \dots < j_n} V_{nND_n}(x_{j_1},\dots,x_{j_n}) 
    &= \frac{1}{n!} 
    \sum_{j_1 \neq \dots \neq j_n} V_{nND_n}(x_{j_1},\dots,x_{j_n}) \\ 
    &= 
    \sum_{j_1, \dots, j_n} V_{nND_n}(x_{j_1},\dots,x_{j_n}) 
    + W_{n-1}, \\     
\end{align*}
where $W_{n-1}$ contains the artificial self-interactions introduced when converting from $\sum_{j_1 \neq \dots \neq j_n}$ to $\sum_{j_1,  \dots, j_n}$. We will return to this term momentarily. 
Now, inserting the expansion of $V_{nND_n}$, writing $c_{\bm v}$ instead of $c^{nND_n}_{\bm v}$, and interchanging summation order, we obtain 
\begin{align*}
    \sum_{j_1, \dots, j_n = 1}^M V_{nND_n}(x_{j_1},\dots,x_{j_n})
    &= 
    \sum_{j_1, \dots, j_n = 1}^M 
    \sum_{v_1, \dots, v_n} c_{\bm v} \prod_{t = 1}^n \phi_{v_t}(x_{j_t}) \\ 
    &= 
    \sum_{v_1, \dots, v_n} c_{\bm v} \prod_{t = 1}^n \sum_{j = 1}^M \phi_{v_t}(x_j) \\ 
    &=
    \sum_{v_1, \dots, v_n} c_{\bm v} \prod_{t = 1}^n A_v({\bm x}),
\end{align*}
where the inner-most terms (power sum polynomials, atomic base, density projection) are now computed over the full input range $x_1, \dots x_M$ instead of only a subcluster, 
\[
    A_{v}({\bm x}) := \sum_{j = 1}^M \phi_{v}(x_{j}).
\]
The self-interaction terms $W_{n-1}$ are simply polynomials of lower correlation-order and can be absorbed into the $\mathcal{O}(n-1)$ terms, provided that 
$D_n \leq D_{n-1}$. Thus, we can equivalently write \eqref{eq:defn_cluster_expansion} as
\begin{equation} \label{eq:efficient_cluster_expansion}
    f_{N{\bm D}} = 
    \sum_{n = 0}^N \sum_{{\bm v}} \tilde{c}_{\bm v} \prod_{t = 1}^n A_{v_t}, 
\end{equation}
where the sum $\sum_{{\bm v}}$ ranges over all ordered tuples ${\bm v}$ of length $n$, with $\sum_{t = 1}^n {\rm deg}(\phi_{v_t}) \leq D_n$. 

This leads to the following result, which states that the cost of evaluating $f_{N{\bm D}}$ is the same as evaluating a single instance of each of the components $V_{nND_n}$; that is, the sum over all possible clusters does not incur an additional cost. 

\begin{proposition} \label{th:ace_cost}
    Assume that Assumption~\ref{as:PHI} is satisfied, that the one-particle basis can be evaluated with $\mathcal{O}(1)$ operations per basis function (e.g. via recursion), and that the degrees are decreasing, $D_n \leq D_{n-1}$. Then, cluster expansion $f_{N{\bm D}}$ can be evaluated with at most 
    \begin{equation} \label{eq:ace_cost}
        C \bigg( 
            M D_1^d
            + \mathcal P
        \bigg),
    \end{equation}
    arithmetic operations ($+, \times$), where $\mathcal P \coloneqq \sum_{n = 1}^N P(n,D_n,d)$ is the number of parameters used to represent the $N$ components $V_{nND_n}$ for $n = 1, \dots, N$ and $C$ is some positive constant. 
\end{proposition}

\begin{remark}
    The number of parameters used to represent $V_{nND_n}$ for all $n=1,\dots,N$ can be estimated using the results of the foregoing sections. In particular, this allows us to obtain error estimates in Theorems~\ref{thm:mset-alpha-const} and \ref{thm:mset} in terms of the computational cost.
\end{remark}

\begin{proof}
    Assumption~\ref{as:PHI} implies that there are $O(D_1^d)$ one-particle basis functions, where $D_1 = \max_{1\leq n \leq N} D_n$. Precomputing all density projections $A_v$ therefore requires  $\mathcal{O}(M D_1^d)$ 
    operations. The $n$-correlations $\prod_{t = 1}^n A_{v_t}$ can be evaluated recursively with increasing correlation-order $n$, with $\mathcal{O}(1)$ cost. Hence, the total cost of evaluating \eqref{eq:defn_cluster_expansion} will be as stated in \eqref{eq:ace_cost}.
\end{proof}

\subsection{Error vs Cost estimates: special case}
Having established the remarkably low computational cost of the cluster expansion in the reformulation \eqref{eq:efficient_cluster_expansion}, we can now return to the derivation of error versus cost estimates from~\eqref{eq:error_f_fND}. 
In order to illustrate our main results, we first consider the simplest case and suppose $\alpha_n = \alpha$ is constant.
Motivated by Remark~\ref{rem:cut-off}, we assume $N \ll M \leq N^p$ for some $p > 1$ and define $D_n = D := \lceil c_1 N \log M \rceil$ (independently of $n$), for some $c_1$ to be determined later. Since $D_n$ is constant, the number of parameters for the $N$-correlation contributions will dominate. Since $N \ll D \ll N^{1+1/d}$ 
this puts us into the regime of the integer partition type estimates, which yields the following result. 

\begin{theorem}\label{thm:mset-alpha-const}
    Assume that $\alpha_n = \alpha$ appearing in~\eqref{eq:VnND-estimate} is constant and $p \geq 1$. If we choose $D_n = D = c_1 N \log N$ 
    for $c_1$ sufficiently large, then 
    \begin{equation}
    \label{eq:thm15:a}
          |f(\bm x) - f_{N\bm{D}}(\bm x)| \lesssim e^{- \eta N}
        \qquad \forall \bm x = [x_1,\dots,x_M] \in \mathrm{MS}(\Omega) \text{ with } M \leq N^p. 
    \end{equation}
    In terms of the total number of free parameters, $\mathcal P = \sum_{n = 1}^N P(n,D_n,d)$, which is also directly proportional to the computational cost, the estimate reads 
    \begin{equation}
        \label{eq:thm15:b}
         |f(\bm x) - f_{N\bm D}(\bm x)| \lesssim \exp\Big( 
            - \tilde\eta \frac{ [\log \mathcal P]^{1+1 / d} }{\log \log \mathcal P} \Big),
    \end{equation}
    for all $\bm x = [x_1,\dots,x_M] \in \mathrm{MS}(\Omega)$ with $M \leq N^p$, for some $\tilde\eta>0$,
    which is still a super-alebraic rate of convergence.
\end{theorem}

\begin{proof}
    Applying (\ref{eq:error_f_fND}) directly with $D_n = D$ for $1\leq n\leq N$, we obtain
    \begin{align*}
        |f(\bm x) - f_{N\bm{D}}(\bm x)| 
        &\lesssim e^{-\eta N} + \sum_{n=1}^N {M \choose n} c^n N^{-\alpha c_1 N}\\
        &\leq e^{-\eta N} + N  ({\max\{ 1, c \}}N^p)^N N^{-\alpha c_1 N}.
    \end{align*}
    Therefore, we obtain the desired estimate~\eqref{eq:thm15:a} by choosing $c_1>0$ sufficiently large.
    
    Using (\ref{eq:hardyramanujan_multiD}), there exists $c_2,c_3 > 0$ such that
    \[
        \mathcal P = \sum_{n=1}^N P(n,D,d) \leq c_1 N^2 \log N e^{c_2 (c_1 N\log N)^{\frac{d}{d+1}}}
        \leq e^{c_3 (N \log N)^{\frac{d}{d+1}}}.
    \]
    That is, $ c_3^{\frac{d+1}{d}} N\log N \geq [\log \mathcal P]^{1 + \frac{1}{d}}$. Now, if $a = N \log N$, then $N = \frac{\log a}{\log N}\frac{a}{\log a} = \big( 1 + \frac{\log\log N}{\log N}\big) \frac{a}{\log a}$. Therefore, since $t \mapsto \frac{\log\log t}{\log t}$ is bounded for $t > 1 + \delta$ (for all $\delta > 0$) and $t \mapsto \frac{t}{\log t}$ is increasing for $t > e$, there exists $c_4 > 0$ such that 
    $
        N \geq c_4\frac{[\log \mathcal P]^{1 + \frac{1}{d}}}{\log \log \mathcal P}
    $ which proves~\eqref{eq:thm15:b}.
\end{proof}

\subsection{Error vs Cost estimates: General Case}

To make this analysis concrete we will assume that in~\eqref{eq:VnND-estimate}
\begin{equation}
    \alpha_n = \alpha_1 n^\beta, \qquad \text{for some } \beta \in (0,1).
\end{equation}
For the sake of generality, we consider this more general class, which better highlights the importance of balancing degree $D_n$ and dimensionality $n$, and also motivates us to revisit and try to sharpen the analysis of Appendix~\ref{app:body-order-assumption} and \cite{ThomasChenOrtner2021:body-order} in the future. In light of Proposition~\ref{th:ace_cost} we will formulate the result in terms of the number of parameters rather than in terms of computational cost. 

\begin{theorem}\label{thm:mset}
    Assume that $\alpha_n = \alpha_1 n^\beta$ for some $\beta \in (0,1)$ and $p \geq 1$. If we choose 
    \begin{equation}
    \label{eq:thm16:0}
        D_n = c_1\begin{cases}
        n^{-\beta}N &\text{if } n \leq (\log N)^{-\frac{1}{\beta}} N \\
        N^{1-\beta} \log N &\text{otherwise,}         
    \end{cases}
    \end{equation}
    for $c_1$ sufficiently large, then 
    \begin{equation}
        \label{eq:thm16:a}
        \left| f(\bm x) -  f_{N\bm D}(\bm x)\right| \lesssim  e^{-\eta N}
        \qquad \forall \bm x = [x_1,\dots,x_M] \in \mathrm{MS}(\Omega) \text{ with } M \leq N^p.
    \end{equation}
    In terms of the total number of free parameters, $\mathcal P = \sum_{n=1}^N P(n,D_n,d)$, we have, for sufficiently large $N$, 
    \begin{equation}
        \label{eq:thm16:b}
                \left| f(\bm x) -  f_{N\bm D}(\bm x)\right| \lesssim  \exp\left(-\tilde\eta [\log \mathcal P]^{1 + \beta + \frac{1}{d}}\right)
    \end{equation}
    for all $\bm x = [x_1,\dots,x_M] \in \mathrm{MS}(\Omega)$ with $M \leq N^p$, for some $\tilde\eta>0$. 
\end{theorem}

\begin{proof}
\textit{Part 1: Error estimates~\eqref{eq:thm16:a}.} We balance the error by choosing $\bm D = (D_n)$ such that $\displaystyle {M \choose n} c^n e^{-\alpha_1 n^\beta D_n} \leq \frac{1}{N} e^{-\eta N}$ for $n=1,\dots,N$ where $c>0$ is the constant from Assumption~\ref{as:body-order}. 
To do so, we first apply Stirling's formula, 
which gives for $n \geq 1$
 \begin{align*}
        {M \choose n} c^n e^{-\alpha_1 n^\beta D_n} 
        = \frac{M!}{n!(M-n)!} c^n e^{-\alpha_1 n^\beta D_n} 
        \leq \left(\frac{ceM}{n}\right)^n e^{-\alpha_1 n^\beta D_n}.
    \end{align*}
Using that $e^{-N} \le 1/N$ for $N \ge 0$, we therefore choose $D_n$ such that 
    \begin{align*}
        \left(\frac{ceM}{n}\right)^n e^{-\alpha_1 n^\beta D_n} 
        \leq e^{-(\eta + 1) N}.
    \end{align*}
That is, to obtain the required estimate, noting that $n\le N$, it is enough to  choose 
\[
\alpha_1 n^\beta D_n
    \geq n \log \frac{M}{n} + (\eta + \log c + 2) N. 
\]

Since $\log \frac{M}{n} \leq \log \frac{N^p}{n} \leq p \log N$ for all $1 \leq n \leq N$, we may instead choose $(D_n)$ such that
\[
    D_n \geq c_1 n^{-\beta} \max\left\{ n\log N, N \right\} 
\]
where $c_1 := \alpha_1^{-1}( p + \eta + \log c + 2 )$. Since $n \log N \leq N$ for all $1 \leq n \leq (\log N)^{-1} N$ and $n \log N \geq N$ for all $n \geq (\log N)^{-1} N$, and by incorporating the condition that $(D_n)$ is decreasing, we may choose 
\[
    D_n := c_1 \max_{k\geq n} \begin{cases}
        k^{-\beta} N &\text{if } 1 \leq k \leq (\log N)^{-1} N \\
        k^{1-\beta} \log N &\text{if } (\log N)^{-1} N \leq k \leq N.
    \end{cases}
\] 

To conclude, we note that, for $\beta \in (0,1)$, the function $k \mapsto k^{-\beta} N$ is decreasing and $k \mapsto k^{1-\beta} \log N$ is increasing and so there exists $1\leq n^\star \leq (\log N)^{-1}N$ for which $D_n = c_1 n^{-\beta} N$ for all $n \leq n^\star$ and $D_n = c_1 N^{1-\beta}\log N$ for all $n\geq n^\star$. Solving $(n^\star)^{-\beta} N = N^{1-\beta} \log N$ yields $n^\star = (\log N)^{-\frac{1}{\beta}} N$ as required. This concludes the proof of~\eqref{eq:thm16:a}.

\begin{remark}Moreover, by the same arguments, if $\beta \leq 0$, then we may choose $D_n = c_1 N^{1-\beta} \log N$ for all $1 \leq n \leq N$ (which, in particular, agrees with the choice in Theorem~\ref{thm:mset-alpha-const}), and if $\beta \geq 1$, we can choose
\[
    D_n = c_1\begin{cases}
        n^{-\beta}N &\text{if } n \leq (\log N)^{-1} N \\
        n^{1-\beta} \log N &\text{otherwise.}         
    \end{cases}
\]
\end{remark}

\textit{Part 2: Estimates on the number of parameters~\eqref{eq:thm16:b}.} Let $P_n \coloneqq P(n,D_n,d)$ be the number of parameters needed to construct $V_{nND_n}$. According to~\eqref{thm:hardyramanujan_multiD}, there exist $c_2,c_3 > 0$ such that
    \begin{align*}
        P_n &\leq nc_2^n D_n \frac{(D_n + c_3 n^{\frac{d+1}{d}})^{dn}}{(dn)! n!}. 
    \end{align*}
            Using Stirling's estimate, we obtain
            \[
                P_n \le 
                \frac
            {(c_2 e^{d+1})^n D_n}
            {2\pi d^{dn+\frac{1}{2}} n^{(d+1)n}} \; 
        2^{dn} \max\{ D_n, c_3 n^{\frac{d+1}{d}}\}^{dn}.
            \]
            Hence, for some $c_4 > 0$,
        \begin{align*}
        P_n &\leq e^{c_4 n} D_n 
        \max\left\{\left( \tfrac{D_n^{d}}{n^{d+1}}\right)^n, 1\right\}
        %
        = \begin{cases}
            D_n \left(\frac{e^{c_4} D_n^d}{n^{d+1}}\right)^n
            &\text{if } 1 \leq n \leq D_n^{\frac{d}{d+1}} 
            \\
            D_n e^{c_4 n} 
            &\text{if } D_n^{\frac{d}{d+1}} \leq n \leq N
        \end{cases}
    \end{align*}
    On the other hand, by \eqref{eq:hardyramanujan_multiD}, there exists $c_5 > 0$ such that 
    \begin{equation}
        \label{eq:recall18}
        P_{n} \leq D_n \mathrm{exp}\big[ c_5 D_n^{\frac{d}{d+1}} \big],
    \end{equation}
    for all $1 \leq n \leq N$.

    \textit{Case (i): $1 \leq n \leq (\log N)^{-\frac{1}{\beta}} N$.} In this case, we have 
    following~\eqref{eq:thm16:0} $D_n = c_1 n^{-\beta} N$
    and so $n \leq D_n^{\frac{d}{d+1}}$ if and only if $n \leq (c_1 N)^{\frac{1}{1 + \beta + 1/d}}$. Therefore, assuming $N$ is sufficiently large such that
    \[
        (c_1N)^{\frac{1}{1+\beta+1/d}}
        \leq (\log N)^{-\frac{1}{\beta}}N,
    \]
     we obtain
    \begin{align}\label{eq:Pn-1}
        P_n \leq D_n \left(\frac{e^{c_4} D_n^d}{n^{d+1}}\right)^n \leq 
        c_1 {n^{-\beta}} N \left(\frac{e^{c_4} (c_1N)^d}{n^{1 + \beta d+ d}}\right)^n
        \qquad 
        \textrm{for } 1 \leq n \leq (c_1N)^{\frac{1}{1+\beta+1/d}}.
    \end{align}
    
    Moreover, in the case $(c_1N)^{\frac{1}{1+\beta+1/d}} \leq n \leq (\log N)^{-\frac{1}{\beta}} N$, we have
    \[
        c_5 D_n^{\frac{d}{d+1}} = 
        c_5 (c_1 n^{-\beta} N )^{\frac{d}{d + 1}} 
        \leq c_5 \left[ (c_1N)^{1-\frac{\beta}{1+\beta+1/d}}
        \right]^{\frac{d}{d + 1}}
        = c_5 (c_1N)^{\frac{1}{1 + \beta + 1/d}}
    \]
    and thus using~\eqref{eq:recall18} 
    \begin{align}\label{eq:Pn-2}
        P_n \leq D_n e^{c_5 D_n^{\frac{d}{d+1}}}
        \leq c_1 n^{-\beta} N e^{c_5 (c_1N)^{\frac{1}{1+\beta + 1/d}}}
        \qquad \textrm{for } (c_1N)^{\frac{1}{1+\beta+1/d}} \leq n \leq {(\log N)^{-\frac{1}{\beta}}N}.
    \end{align}
    
    \textit{Case (ii): $(\log N)^{-\frac{1}{\beta}} N\leq n \leq N$.} In this case, $D_n = c_1 N^{1-\beta} \log N$ and so using~\eqref{eq:recall18}
    \begin{align}\label{eq:Pn-3}
        P_n \leq D_n e^{c_5 D_n^{\frac{d}{d+1}}}
        \leq  c_1 N^{1-\beta} \log N e^{c_5 (c_1 N^{1-\beta} \log N)^{\frac{d}{d + 1}}}.
    \end{align}
    
    \medskip
    In order to bound the total number of parameters needed to construct $V_{nND_n}$ for $1 \leq n \leq N$, we consider sum of (\ref{eq:Pn-1}), (\ref{eq:Pn-2}), and (\ref{eq:Pn-3}) over their respective ranges of $n$:
    
    We start with (\ref{eq:Pn-1}). Since the function $g(t) := \big(\frac{c}{t^\alpha}\big)^t$ (for $c,\alpha>1$) has a global maximum at $t^\star = \frac{1}{e} c^{1/\alpha}$ with $g(t^\star) = \mathrm{exp}\big[\frac{\alpha}{e} c^\frac{1}{\alpha}\big]$, we have 
    taking $c = e^{c_4} (c_1N)^d$ and $\alpha = 1 + \beta d+ d$
    \begin{align}
        \sum_{n=1}^{\lfloor (c_1 N)^{\frac{1}{1 + \beta + 1/d}} \rfloor} P_n
        &\leq 
        \sum_{n =1}^{\lfloor (c_1 N)^{\frac{1}{1 + \beta + 1/d}} \rfloor}  c_1 n^{-\beta} N \left(\frac{e^{c_4} (c_1N)^d}{n^{1 + \beta d+ d}}\right)^n  \nonumber \\
        &\leq 
        (c_1 N)^{1 + \frac{1}{1 + \beta + 1/d}} \cdot 
        \mathrm{exp}\left[ \frac{1 + \beta d  + d}{e} [e^{c_4}(c_1 N)^d]^{\frac{1}{1+\beta d + d}} \right] \nonumber\\
        &\leq (c_1 N)^{\frac{2+\beta +1/d}{1 + \beta + 1/d}} e^{c_6 N^{\frac{1}{1 + \beta + 1/d}}}\label{eq:sumPn-1}
    \end{align}
    for some $c_6 > 0$. 
    Here, we have used the bound $n^{-\beta} \leq 1$.
    
    Next, we consider the sum of (\ref{eq:Pn-2}). On the relevant range of $n$, we have $n^{-\beta} \leq (c_1N)^{-\frac{\beta}{1+\beta+1/d}}$, and thus
     \begin{align}
        \sum_{n =\lceil (c_1 N)^{\frac{1}{1 + \beta + 1/d}} \rceil}^{(\log N)^{-\frac{1}{\beta}}N} P_n
        &\leq 
        \sum_{n =\lceil (c_1 N)^{\frac{1}{1 + \beta + 1/d}} \rceil}^{(\log N)^{-\frac{1}{\beta}}N} c_1 n^{-\beta} N e^{c_5 (c_1N)^{\frac{1}{1+\beta + 1/d}}} \nonumber \\
        &\leq 
        N (c_1N)^{1 - \frac{\beta}{1+\beta+1/d} } e^{c_5 (c_1N)^{\frac{1}{1+\beta + 1/d}}} \nonumber\\
        &\lesssim N^{\frac{2+ \beta +2/d}{1+\beta+1/d} } e^{c_7 N^{\frac{1}{1 + \beta + 1/d}}}\label{eq:sumPn-3}
    \end{align}
    for some $c_7 > 0$.
    
    Finally, summing (\ref{eq:Pn-3}) gives
    \begin{align}
        \sum_{n = (\log N)^{-\frac{1}{\beta}}N}^N P_n 
        \leq c_1 N^{2-\beta} \log N e^{c_5 (c_1 N^{1-\beta} \log N)^{\frac{d}{d + 1}}}.\label{eq:sumPn-2}
    \end{align}
    
    The dominant contribution (for large $N$) is either (\ref{eq:sumPn-1}) or (\ref{eq:sumPn-3}) since $(1-\beta)\frac{d}{d+1} < \frac{1}{1 + \beta + 1/d}$ for all $\beta \in (0,1)$ and $d \geq 1$. In particular, there exists $c_8 > 0$ such that the total number of parameters needed to construct $V_{nND_n}$ for $n=1,\dots,N$ satisfies $\mathcal P \leq e^{c_8 N^{\frac{1}{1+\beta + 1/d}}}$ and thus
    \[
        N \geq \Big( \frac{1}{c_8} \log \mathcal P \Big)^{1 + \beta + \frac{1}{d}}
    \]
    as required to prove~\eqref{eq:thm16:b}.
\end{proof}

\section{Conclusion}
We have established rigorous approximation rates for sparse and symmetric polynomial approximations of symmetric functions in high dimension. What is particularly intriguing about our results is that they highlight clearly how symmetry reduces the curse of dimensionality, sometimes significantly so. Our results also build a foundation for analysing the approximation of functions defined on a configuration space (multi-sets), which we outlined as well. 

Further open challenges include the incorporation of more complex symmetries, e.g.,  coupling permutation with Lie group symmetries such as $O(d)$ for classical particles or $O(1,d)$ for relativistic particles, as well as the identification of practical constructive algorithms to construct approximants in our context that achieve (close to) the optimal rates.

\section*{Acknowledgements}

M.B.\ acknowledges funding by Deutsche Forschungsgemeinschaft (DFG, German Research Foundation) -- Projektnummer 233630050 -- TRR 146; G. D.'s work was supported by the French ``Investissements d'Avenir'' program, project ISITE-BFC (contract ANR-15-IDEX-0003); C.O.\ is supported by the Natural Sciences and Engineering Research Council of Canada (NSERC) [IDGR019381]; J.T.\ is supported by the Engineering and Physical Sciences Research Council (EPSRC) Grant EP/W522594/1.

\appendix

\renewcommand{\above}[2]{\genfrac{}{}{0pt}{}{#1}{#2}}

\section{Justification of the Many-Body Expansion}
\label{app:body-order-assumption}

The body-ordered approximation asserted in Assumption~\ref{as:body-order} is motivated by recent results presented in \cite{ThomasChenOrtner2021:body-order} for a wide class of tight binding models. In this section, we give a brief justification of these results.

\subsection{Boundedness: $|V_{nN}(x_{j_1},\dots,x_{j_n})| \lesssim 2^n$}

Here, we will briefly note that naive body-ordered approximations satisfy the required bound. For 
$f\colon \ms(\Omega) \to \mathbb R$,
we may define the following \textit{vacuum cluster expansion}: 
\begin{align}
    f_N(\bm x) &:= V_0 + \sum_{n=1}^N \sum_{j_1< \dots < j_n} V_n(x_{j_1},\dots,x_{j_n}) \\
    \textrm{where} \qquad 
    V_n(x_{j_1},\dots,x_{j_n}) &:= \sum_{K \subseteq \{j_1,\dots,j_n\}} (-1)^{n-|K|} f\big([x_{k} \colon k \in K]\big). \label{eq:Vn}
\end{align}
That is, the $n$-body potential is obtained by considering the $n$ variables of interest and removing the terms that correspond to strictly smaller body-order. The alternating summation comes from the inclusion-exclusion principle. By construction, the expansion is exact for systems of size at most $N$ (i.e.~$f([x_1,\dots,x_M]) = f_N([x_1,\dots,x_M])$ for all $M\leq N$). Moreover, assuming $f$ is uniformly bounded, we obtain the required estimate $|V_{nN}(x_{j_1},\dots,x_{j_n})| \lesssim 2^n$.

While this classical vacuum cluster expansion is perhaps the most natural, it is not guaranteed to converge rapidly if at all. Instead, we will review recent results for the site energies for a particular class of simple electronic structure models.

\subsection{Convergence: A tight-binding example}
For an atomic configuration $\bm x = [x_1,\dots,x_M] \in \ms(\Omega)$, a (two-centre) tight binding \textit{Hamiltonian} $\mathcal{H}(\bm x)$, describing the interaction between the atoms in the system, is given by the following matrix
\[
    \mathcal{H}(\bm x)_{ij} := h( x_i - x_j )
\]
for some smooth function $h \colon \mathbb R^d \to \mathbb R$ satisfying $|h(\xi)| + |\nabla h(\xi)| \leq h_0 e^{-\gamma_0 |\xi|}$. 

For functions $o\colon \mathbb R \to \mathbb R$, the corresponding observable, $O(\bm x)$, is written as the following function of the Hamiltonian:
\begin{align}\label{eq:E}
    O(\bm x) 
    := \mathrm{Tr}\, o\big(\mathcal{H}(\bm x)\big) 
    = \sum_{i=1}^M o\big(\mathcal{H}(\bm x)\big)_{ii}.
\end{align}
We will think of $O(\bm x)$ as the total energy of the system and the right-hand side of (\ref{eq:E}) as a site energy decomposition. We will justify Assumption~\ref{as:body-order} for the site energy $f(\bm x) := o\big(\mathcal{H}(\bm x)\big)_{11}$. By approximating $o$ with a polynomial $o_N(x) \coloneqq \sum_{k=0}^{N} c_k x^k$ of degree $N$, we define the body-ordered approximation by
\begin{align}
    f_N(\bm x) 
    &:= o_N\big( \mathcal{H}(\bm x) \big)_{11} 
    = \sum_{k=0}^{N} c_k \big[\mathcal{H}(\bm x)^k\big]_{11} \nonumber\\
    &= \sum_{k=0}^{N} c_k \sum_{i_1,\dots,i_{k-1}} h( x_1 - x_{i_1} ) h( x_{i_1} - x_{i_2} ) \cdots  h( x_{i_{k-2}} - x_{i_{k-1}} )h( x_{i_{k-1}} - x_{1} ),
\end{align}
a function of body-order at most $N$. Convergence results for this approximation scheme follow from the estimate:
\[
    |f(\bm x) - f_N(\bm x)| 
    \leq \big|\big[ o\big(\mathcal{H}(\bm x) \big) - o_N\big( \mathcal{H}(\bm x) \big) \big]_{11} \big|
    \leq \sup_{ z \in \sigma(\mathcal H(\bm x))} | o(z) - o_N(z) |.
\]
Now, if $o$ is an analytic function in a neighbourhood of the spectrum $\sigma\big(\mathcal{H}(\bm x)\big)$, we are able to construct approximations $o_N$ that give an exponential rate of convergence with rate depending on the region of analyticity of $o$ \cite{ThomasChenOrtner2021:body-order}.

Written more explicitly, the body-ordered approximation has a similar expression to that of (\ref{eq:Vn}):
\begin{align}
    f_N(\bm x) &=  V_{0N} + \sum_{n=1}^{N-1} 
    \sum_{ j_1< \dots < j_n } 
    V_{nN}(x_{j_1},\dots,x_{j_n})
    \qquad \text{where} \nonumber\\
    V_{nN}(x_{j_1},\dots,x_{j_{n}}) &:= \sum_{k=0}^{N} c_k 
    \sum_{\above{i_0,i_1\dots,i_{k} \colon i_0=i_k = 1, }{ 
    \{i_0,\dots,i_k\} = \{1,j_1,\dots,j_n\}}} 
   \prod_{l=0}^{k-1} h( x_{j_l} - x_{j_{l+1}} ) \nonumber\\
   &= \sum_{K\subseteq \{j_1,\dots,j_n\}} (-1)^{n-|K|} o_N\big(\mathcal{H}(\bm x)|_{1,K}\big)_{11}. \label{eq:VnN}
\end{align}
Here, $\mathcal{H}(\bm x)|_{1,K}$ is the restriction of $\mathcal{H}(\bm x)$ to $\{1\}\cup K$ defined as $[\mathcal{H}(\bm x)|_{1,K}]_{ij} = \mathcal{H}(\bm x)_{ij}$ if $i,j \in \{1\}\cup K$ and $[\mathcal{H}(\bm x)|_{1,K}]_{ij} = 0$ otherwise. Therefore, in this expression, $V_{nN}(x_{j_1},\dots, x_{j_n})$ is an $(n+1)$-body potential of the central atom $x_1$ and $x_{j_1},\dots, x_{j_n}$ (some authors therefore call $V_{nN}$ an $n$-correlation potential). The final line in (\ref{eq:VnN}) follows from an inclusion-exclusion principle \cite{ThomasChenOrtner2021:body-order}. In particular, the boundedness of the $V_{nN}$ follows from the boundedness of $o_N$ \cite{ThomasChenOrtner2021:body-order}, as in (\ref{eq:Vn}). Moreover, $V_{nN}$ inherits the analyticity properties of the Hamiltonian. 

We have therefore seen that the site energies in the tight binding framework satisfy both the conditions in Assumption~\ref{as:body-order}: \textit{(i)} convergence of a body-ordered approximation, with an exponential rate, and \textit{(ii)} the $V_{nN}$ are analytic (with region of analyticity depending on the regularity of the Hamiltonian).

\bibliographystyle{siam}
\bibliography{biblio}

\begin{thebibliography}{10}

\bibitem{Beged-dov1972-mp}
{\sc A.~G. Beged-dov}, {\em Lower and upper bounds for the number of lattice
  points in a simplex}, SIAM J. Appl. Math., 22 (1972), pp.~106--108.

\bibitem{Chen2016-ig}
{\sc H.~Chen and C.~Ortner}, {\em {QM/MM} methods for crystalline defects. part
  1: Locality of the tight binding model}, Multiscale Model. Simul., 14 (2016),
  pp.~232--264.

\bibitem{Cohen2015-ol}
{\sc A.~Cohen and R.~DeVore}, {\em Approximation of high-dimensional parametric
  {PDEs} *}, Acta Numer., 24 (2015), pp.~1--159.

\bibitem{Derksen2015-km}
{\sc H.~Derksen and G.~Kemper}, {\em Computational Invariant Theory}, Springer,
  Dec. 2015.

\bibitem{Drautz2019-rb}
{\sc R.~Drautz}, {\em Atomic cluster expansion for accurate and transferable
  interatomic potentials}, Phys. Rev. B Condens. Matter, 99 (2019), p.~014104.

\bibitem{acewave2022}
{\sc R.~Drautz and C.~Ortner}.
\newblock in preparation.

\bibitem{Bachmayr2019-ec}
{\sc G.~Dusson, M.~Bachmayr, G.~Csanyi, R.~Drautz, S.~Etter, C.~van~der Oord,
  and C.~Ortner}, {\em Atomic cluster expansion: Completeness, efficiency and
  stability}, J. Comp. Phys., 454 (2022).

\bibitem{Erdos1942-qa}
{\sc P.~Erdos}, {\em On an elementary proof of some asymptotic formulas in the
  theory of partitions}, Ann. Math., 43 (1942), pp.~437--450.

\bibitem{GriebelOettershagen:16}
{\sc M.~Griebel and J.~Oettershagen}, {\em On tensor product approximation of
  analytic functions}, Journal of Approximation Theory, 207 (2016),
  pp.~348--379.

\bibitem{Han2019-ae}
{\sc J.~Han, Y.~Li, L.~Lin, J.~Lu, J.~Zhang, and L.~Zhang}, {\em Universal
  approximation of symmetric and anti-symmetric functions}, arXiv e-prints,
  1912.01765 (2019).

\bibitem{Hardy1918-gg}
{\sc G.~H. Hardy and S.~Ramanujan}, {\em Asymptotic formula{\ae} in combinatory
  analysis}, Proceedings of the London Mathematical Society,  (1918).

\bibitem{Kall2021}
{\sc I.~Kaliuzhnyi and C.~Ortner}, {\em Optimal evaluation of symmetry-adapted
  $n$-correlations via recursive contraction of sparse symmetric tensors}.
\newblock arXiv:2202.04140.

\bibitem{Mason1980-li}
{\sc J.~C. Mason}, {\em Near-best multivariate approximation by fourier series,
  chebyshev series and chebyshev interpolation}, J. Approx. Theory, 28 (1980),
  pp.~349--358.

\bibitem{NovakWozniakowski08}
{\sc E.~Novak and H.~Wo\'{z}niakowski}, {\em Tractability of multivariate
  problems. {V}ol. 1: {L}inear information}, vol.~6 of EMS Tracts in
  Mathematics, European Mathematical Society (EMS), Z\"{u}rich, 2008.

\bibitem{Qi2017}
{\sc C.~R. Qi, H.~Su, K.~Mo, and L.~J. Guibas}, {\em Pointnet: Deep learning on
  point sets for 3d classification and segmentation}, in Proceedings of the
  IEEE Conference on Computer Vision and Pattern Recognition (CVPR), July 2017.

\bibitem{2011-ip}
{\sc {R B J} and A.~Slomson}, {\em How to Count: An Introduction to
  Combinatorics, Second Edition}, CRC Press, July 2011.

\bibitem{Ramirez_Alfonsin2005-et}
{\sc J.~L. Ram{\'\i}rez~Alfons{\'\i}n}, {\em The Diophantine Frobenius
  Problem}, OUP Oxford, Dec. 2005.

\bibitem{Ryan2002}
{\sc R.~A. Ryan}, {\em Introduction to Tensor Products of Banach Spaces},
  Springer Monographs in Mathematics, Springer London, 2002.

\bibitem{sagan2001symmetric}
{\sc B.~Sagan}, {\em The symmetric group: representations, combinatorial
  algorithms, and symmetric functions}, vol.~203, Springer Science \& Business
  Media, 2001.

\bibitem{ThomasChenOrtner2021:body-order}
{\sc J.~Thomas, H.~Chen, and C.~Ortner}, {\em Rigorous body-order
  approximations of an electronic structure potential energy landscape}, arXiv
  e-prints, 2106.12572 (2021).

\bibitem{Trefethen2017-rc}
{\sc L.~Trefethen}, {\em Multivariate polynomial approximation in the
  hypercube}, Proc. Am. Math. Soc.,  (2017).

\bibitem{Wagstaff2019}
{\sc E.~Wagstaff, F.~Fuchs, M.~Engelcke, I.~Posner, and M.~A. Osborne}, {\em On
  the limitations of representing functions on sets}, in Proceedings of the
  36th International Conference on Machine Learning, K.~Chaudhuri and
  R.~Salakhutdinov, eds., vol.~97 of Proceedings of Machine Learning Research,
  PMLR, 09--15 Jun 2019, pp.~6487--6494.

\bibitem{Weimar12}
{\sc M.~Weimar}, {\em The complexity of linear tensor product problems in
  (anti)symmetric {H}ilbert spaces}, J. Approx. Theory, 164 (2012),
  pp.~1345--1368.

\bibitem{Yutsis1965-rr}
{\sc A.~P. Yutsis and A.~A. Bandzaitis}, {\em Theory of angular momentum in
  quantum mechanics}, Vil'nyus,  (1965).

\bibitem{Zaheer2017:deepsets}
{\sc M.~Zaheer, S.~Kottur, S.~Ravanbakhsh, B.~Poczos, R.~R. Salakhutdinov, and
  A.~J. Smola}, {\em Deep sets}, in Advances in Neural Information Processing
  Systems, I.~Guyon, U.~V. Luxburg, S.~Bengio, H.~Wallach, R.~Fergus,
  S.~Vishwanathan, and R.~Garnett, eds., vol.~30, Curran Associates, Inc.,
  2017.

\end{thebibliography}

\end{document}